\documentclass[12pt]{amsart}
\usepackage{amsmath,amssymb,latexsym,cancel,rotating}
\usepackage{graphicx,mathrsfs,color,fancyhdr,amsthm}
\usepackage{tikz}
\usepackage{geometry}
\usepackage{hyperref}
\usepackage{diagrams}
\newtheorem{theorem}{Theorem}[section]
\newtheorem{lemma}[theorem]{Lemma}
\newtheorem{corollary}[theorem]{Corollary}
\newtheorem{definition}[theorem]{Definition}
\newtheorem{proposition}[theorem]{Proposition}

\newtheorem{example}[theorem]{Example}

\def\cP{{\mathcal P}} 
\def\cC{{\mathcal C}} \def\cD{{\mathcal D}}    
\def\cH{{\mathcal H}}     
 \def\cO{{\mathcal O}} 
    \def\cQ{{\mathcal Q}}  
\def\cA{{\mathcal A}}    
\def\cK{{\mathcal K}}
\def\bbN{{\mathbb N}}  \def\bbZ{{\mathbb Z}}  \def\bbQ{{\mathbb Q}}
    \def\bbF{{\mathbb F}}
\def\bbC{{\mathbb C}}

         \def\bfU{{\bf U}}

\def\Hom{\mbox{\rm Hom}}  
       
\def\Ext{\mbox{\rm Ext}}   
\def\dim{\mbox{\rm dim}}   
\def\Aut{\mbox{\rm Aut}}
 
\def\mod{\mbox{\rm mod}}  \def\tr{\mbox{\rm tr}}
  
\def\rep{\mbox{\rm Rep}\,}

\def\bfV{{\mathbf V}}

\def\rep{{\rm rep}}

\def\bfV{\mathbf{V}}
\def\bfW{\mathbf{W}}
\def\bfE{\mathbf{E}}
\def\bfG{\mathbf{G}}

\def\bfQ{\mathbf{Q}}
\def\bfU{\mathbf{U}}
\def\bfF{\mathbf{F}}
\def\Ind{\mathbf{Ind}}
\def\Res{\mathbf{Res}}
\def\ind{\mathbf{ind}}
\def\res{\mathbf{res}}
\def\cN{\mathcal{N}}

\begin{document}

\title[]{The parity of Lusztig's restriction functor and Green's formula for a quiver with automorphism}

\author{Jiepeng Fang, Yixin Lan, Yumeng Wu}
\address{School of Mathematical Sciences, Peking University, Beijing 100871, P.R. China}
\email{fangjp@math.pku.edu.cn (J. Fang)}
\address{Academy of Mathematics and Systems Science, Chinese Academy of Sciences, Beijing 100190, P.R. China}
\email{lanyixin@amss.ac.cn (Y. Lan)}
\address{Department of Mathematical Sciences, Tsinghua University, Beijing 100084, P. R. China}
\email{wuym20@mails.tsinghua.edu.cn (Y. Wu)}



\bibliographystyle{abbrv}

\maketitle

\begin{abstract}
In \cite{fang2023parity}, Fang-Lan-Xiao proved a formula about Lusztig's induction and restriction functors which can induce Green's formula for the path algebra of a quiver over a finite field via the trace map. In this paper, we generalize their formula to that for the mixed semisimple perverse sheaves for a quiver with an automorphism. By applying the trace map, we obtain Green's formula for any finite-dimensional hereditary algebra over a finite field. 
\end{abstract}

\setcounter{tocdepth}{1}\tableofcontents

\section{Introduction}

In \cite{ringel1990hall,Ringel-1990}, Ringel introduced the Hall algebra for any finite-dimensional hereditary algebra $\Lambda$ over a finite field $\bbF_q$. More precisely, the Hall algebra $\cH(\mod_{\bbF_q}\Lambda)$ is a $\bbC$-algebra with a basis $\{u_{[M]}|[M]\in \textrm{Iso}(\mod_{\bbF_q}\Lambda)\}$ parametrized by the isomorphism classes of finitely generated right $\Lambda$-modules over $\bbF_q$ and the multiplication
$$u_{[M]}*u_{[N]}=\sum_{[L]}\sqrt{q}^{\langle M,N\rangle}g^L_{MN}u_{[L]},$$
where $\langle M,N\rangle$ is the Euler form and $g^L_{MN}$ is the number of submodules $L'$ of $L$ satisfying $L/L'\cong M,L'\cong N$. In particular, if $\Lambda$ is of finite type, he proved that the Hall algebra is isomorphic to the specialization of the nilpotent part ${\rm U}_v^-$ at $v=\sqrt{q}$ of the quantum group defined by the Cartan datum associated to $\Lambda$.

In \cite{green1995hall}, Green equipped the Hall algebra $\cH(\mod_{\bbF_q}\Lambda)$ with a comultiplication 
$$\Delta(u_{[L]})=\sum_{[M],[N]}\sqrt{q}^{\langle M,N\rangle}g^L_{MN}a_Ma_Na_L^{-1}u_{[M]}\otimes u_{[N]},$$
where $a_{M}=|\Aut_\Lambda(M)|$ is the order of the automorphism group, and $a_N,a_L$ are similar. He proved the following non-trivial homological formula 
\begin{equation}\label{introduction Green's formula}
\begin{aligned}
&a_Ma_Na_{M'}a_{N'}\sum_{[L]}g^L_{MN}g^L_{M'N'}a_L^{-1}\\
=&\sum_{[M_1],[M_2],[N_1],[N_2]}q^{-\langle\hat{M_1},\hat{N_2} \rangle}g^M_{M_1M_2}g^N_{N_1N_2}g^{M'}_{M_1N_1}g^{N'}_{M_2N_2}a_{M_1}a_{M_2}a_{N_1}a_{N_2},
\end{aligned}
\end{equation}
for any $M,N,M',N'\in \mod_{\bbF_q}\Lambda$, equivalently, 
\begin{align}\label{equivalent formula}
\Delta(u_{[M]}*u_{[N]})=\Delta(u_{[M]})*\Delta(u_{[N]})
\end{align}
for any $u_{[M]},u_{[N]}\in \cH(\mod_{\bbF_q}\Lambda)$, such that $\Delta$ is an algebra homomorphism with respect to a twisted multiplication on $\cH(\mod_{\bbF_q}\Lambda)\otimes \cH(\mod_{\bbF_q}\Lambda)$, see Theorem \ref{Green's formula} for details. Moreover, he proved that the composition subalgebra $\cC(\mod_{\bbF_q}\Lambda)$ generated by elements corresponding to simple modules is isomorphic to the specialization of ${\rm U}_v^-$ at $v=\sqrt{q}$ as bialgebras. 

In \cite{Lusztig-1991, lusztig2010introduction}, Lusztig established the framework of the geometric realization for Hall algebra and categorified ${\rm U}_v^-$ by perverse sheaves. For any finite quiver without loops $Q$, he considered the affine variety $\bfE_\nu$ together with the connected algebraic group $\bfG_\nu$-action which parametrizes the isomorphism classes of representations of $Q$ having dimension vector $\nu$ over the algebraic closure $k=\overline{\bbF}_q$, and defined a subcategory $\cQ_\nu$ of $\cD^b_{\bfG_\nu}(\bfE_\nu)$ the $\bfG_\nu$-equivariant bounded derived category of $\overline{\bbQ}_l$-constructible sheaves on $\bfE_\nu$ consisting of the direct sums of simple perverse sheaves which appear as direct summands in the direct images of constant sheaves on the flag varieties up to shifts and Tate twists. For any $\nu,\nu',\nu''$ satisfying $\nu=\nu'+\nu''$, he defined the induction and the restriction functors 
\begin{align*}
&\Ind^{\nu}_{\nu',\nu''}:\cD^b_{\bfG_{\nu'}}(\bfE_{\nu'})\boxtimes \cD^b_{\bfG_{\nu''}}(\bfE_{\nu''})\rightarrow \cD^b_{\bfG_\nu}(\bfE_\nu),\\
&\Res^{\nu}_{\nu',\nu''}:\cD^b_{\bfG_\nu}(\bfE_\nu)\rightarrow \cD^b_{\bfG_{\nu'}\times \bfG_{\nu''}}(\bfE_{\nu'}\times \bfE_{\nu''})
\end{align*}
and proved that $\Ind^{\nu}_{\nu',\nu''}(\cQ_{\nu'}\boxtimes \cQ_{\nu''})\subset \cQ_{\nu}, \Res^{\nu}_{\nu',\nu''}(\cQ_\nu)\subset \cQ_{\nu'}\boxtimes \cQ_{\nu''}$ such that all induction functors induce a multiplication and all restriction functors induce a comultiplication on $\cK=\bigoplus_\nu\cK(\cQ_\nu)$ the direct sum of Grothendieck groups which has a $\bbZ[v,v^{-1}]$-module structure given by $v^{-1}.[A]=[A[1](\frac{1}{2})]$. Moreover, he proved the following formula
\begin{equation}\label{Lusztig's formula}
\begin{aligned}
\Res_{\alpha' ,\beta'}^{\gamma}\Ind_{\alpha,\beta}^{\gamma}(A\boxtimes B)\cong \bigoplus_{\lambda=(\alpha_1,\alpha_2,\beta_1,\beta_2) \in \cN}(\Ind_{\alpha_1,\beta_1}^{\alpha'}\boxtimes \Ind_{\alpha_2,\beta_2}^{\beta'}){\tau_\lambda}_!\\
(\Res_{\alpha_1,\alpha_2}^{\alpha}(A)\boxtimes \Res_{\beta_1,\beta_2}^{\beta}(B))[-(\alpha_2,\beta_1)](-\frac{(\alpha_2,\beta_1)}{2})
\end{aligned}
\end{equation}
for any $A\in \cQ_{\alpha}, B\in \cQ_\beta$, see \cite[Proposition 8.4]{Lusztig-1991} or \cite[Theorem 3.1]{fang2023parity} for details, such that $\cK$ is a bialgebra, and proved that $\bbC(v)\otimes_{\bbZ[v,v^{-1}]}\cK\cong {\rm U}_v^-$ as bialgebras.

The connection between the Hall algebra and Lusztig's geometric construction is given by the sheaf-function correspondence, see \cite{lusztig1998canonical,xiao2019ringel,fang2023parity}. The trace map induces a bialgebra homomorphism from $\bbC(v)\otimes_{\bbZ[v,v^{-1}]}\cK$ onto the composition subalgebra $\cC(\rep_{\bbF_q}(Q))$. As a consequence, by applying the trace map to the formula (\ref{Lusztig's formula}), one can obtain the comultiplication formula (\ref{equivalent formula}) for any $u_{[M]},u_{[N]}\in \cC(\rep_{\bbF_q}(Q))$.

In \cite{fang2023parity}, Fang-Lan-Xiao proved the formula (\ref{Lusztig's formula}) for any equivariant mixed semisimple complexes $A\in \cD^{b,ss}_{\bfG_\alpha,m}(\bfE_\alpha),B\in \cD^{b,ss}_{\bfG_\beta,m}(\bfE_\beta)$. By applying the trace map, one can obtain the comultiplication formula (\ref{equivalent formula}) for any $u_{[M]},u_{[N]}\in \cH(\rep_{\bbF_q}(Q))$ beyond the composition subalgebra, equivalently, Green's formula (\ref{introduction Green's formula}) for any $M,N\in\rep_{\bbF_q}(Q)$.

In the present paper, we generalize Fang-Lan-Xiao's result and prove the following theorem, see section \ref{Main theorem and application} for details.
\begin{theorem}
For any $(A,\varphi)\in \tilde{\cD}_\alpha$ and $(B,\psi) \in \tilde{\cD}_{\beta}$, we have 
\begin{align*}
\Res_{\alpha' ,\beta'}^{\gamma}\Ind_{\alpha,\beta}^{\gamma}((A,\varphi)\boxtimes(B,\psi))\cong (C,\phi)\oplus \bigoplus_{\lambda=(\alpha_1,\alpha_2,\beta_1,\beta_2) \in \cN^a}(\Ind_{\alpha_1,\beta_1}^{\alpha'}\boxtimes \Ind_{\alpha_2,\beta_2}^{\beta'})\tilde{{\tau_\lambda}_!}\\
(\Res_{\alpha_1,\alpha_2}^{\alpha}(A,\varphi)\boxtimes \Res_{\beta_1,\beta_2}^{\beta}(B,\psi))[-(\alpha_2,\beta_1)](-\frac{(\alpha_2,\beta_1)}{2}),
\end{align*}
where $(C,\phi)\in \tilde{\cD}_{\alpha',\beta'}$ is traceless.
\end{theorem}
This formula is a generation of the formula (\ref{Lusztig's formula}) for equivariant mixed semisimple complexes for the quiver $Q$ with the admissible automorphism $a$. As an application, by applying the trace map, we obtain Green's formula (\ref{introduction Green's formula}) for the Hall algebra $\cH(\rep_k^{F_{Q,a}}(Q))$ of $F_{Q,a}$-stable representations, see section \ref{Hereditary algebra and quiver with automorphism} for definitions. Moreover, for any finite-dimensional hereditary basic algebra $\Lambda$ over $\bbF_q$, there is a finite quiver $Q_\Lambda$ without loops and an admissible automorphism $a$ such that 
$$\mod_{\bbF_q}\Lambda\simeq \rep_k^{F_{Q_\Lambda,a}}(Q_\Lambda),\ \cH(\mod_{\bbF_q}\Lambda)\cong \cH(\rep_k^{F_{Q_\Lambda,a}}(Q_\Lambda)),$$ see section \ref{Hall algebra}. Therefore, we also obtain Green's formula (\ref{introduction Green's formula}) for $\mod_{\bbF_q}\Lambda$.

In section \ref{Hereditary algebra and quiver with automorphism}, we review the definition of the quiver with an automorphism and its relation with the finite-dimensional hereditary basic algebra over the finite field, and refer \cite{deng2006frobenius, Deng-Du-Parshall-Wang-2008} for more details. In section \ref{Hall algebra}, we review the definition of Hall algebra for certain abelian category, and refer \cite{ringel1990hall, green1995hall, Schiffmann-2012} for details. In section \ref{Induction functor and restriction functor}, we review some preliminary about the periodic functor, the equivariant mixed semisimple complex and the trace map, and review Lusztig's geometric framework including the induction and the restriction functors, then give the connection between the geometric construction and the Hall algebra by sheaf-function correspondence. In section \ref{Main theorem and application}, we prove the main theorem and deduce Green's formula as an application. 

\subsection*{Convention} 
Throughout this paper, $p,l$ are two fixed distinct prime number, $\bbF_q$ is the finite field of order $q=p^n$ for some $n\geqslant 1$, $k=\overline{\bbF}_q$ is its algebraic closure and $\overline{\bbQ}_l$ is the algebraic closure of the field of $l$-adic numbers. For any field $K$ and any finite-dimensional algebra $A$ over $K$, we denote by $\mod_KA$ the category of finitely generated right $A$-modules over $K$. For any finite set $X$, we denote by $|X|$ its order.

\section{Quiver with automorphism and hereditary algebra}\label{Hereditary algebra and quiver with automorphism}

\begin{definition}\label{quiver with automorphism}
{\rm{(1)}} A finite quiver $(I,H,s,t)$ consists of two finite sets $I,H$ and two maps $s,t:H\rightarrow I$, where $I$ is the set of vertices, $H$ is the the set of arrows, and for any $h\in H$,  the images $s(h)$ and $t(h)\in I$ are its source and target respectively. A loop of the quiver is an arrow $h\in H$ satisfying $s(h)=t(h)$.

{\rm{(2)}} Let $(I,H,s,t)$ be a finite quiver without loops, an admissible automorphism $a$ of the quiver consists of two permutations $a:I\rightarrow I$ and $a:H\rightarrow H$ satisfying\\ 
{\rm{(a)}} $a(s(h))=s(a(h)), a(t(h))=t(a(h))$ for any $h\in H$;\\
{\rm{(b)}} $s(h),t(h)\in I$ belong to different $a$-orbits for any $h\in H$.
\end{definition}

\begin{definition}
{\rm{(1)}} Let $V$ be a $k$-vector space, a Frobenius map on $V$ is a $\bbF_q$-linear isomorphism $F_V:V\rightarrow V$ satisfying\\
{\rm{(a)}} for any $\lambda\in k$ and $v\in V$, we have $F_V(\lambda v)=\lambda^qF_V(v)$;\\
{\rm{(b)}} for any $v\in V$, there exists $n\geqslant 1$ such that $F_V^n(v)=v$.\\
If there exists such a Frobenius map, then the fixed point set $V^{F_V}$ is a $\bbF_q$-subspace such that $V=k\otimes_{\bbF_q}V^{F_V}$, and we say $V$ has a $\bbF_q$-structure.

{\rm{(2)}} Let $A$ be an algebra over $k$, a Frobenius morphism on $A$ is a Frobenius map $F_A:A\rightarrow A$ on the underlying $k$-vector space preserving the unit and the multiplication. If there exists such a Frobenius morphism, then the fixed point set $A^{F_A}$ is a $\bbF_q$-subalgebra such that $A=k\otimes_{\bbF_q}A^{F_A}$, and we say $A$ has a $\bbF_q$-structure.

{\rm{(3)}} Let $A$ be an algebra over $k$ with the Frobenius morphism $F_A:A\rightarrow A$ and $M\in \mod_k A$, a Frobenius morphism on $M$ is a Frobenius map $F_M:M\rightarrow M$ on the underlying $k$-vector space satisfying $F_M(ma)=F_M(m)F_A(a)$ for any $m\in M$ and $a\in A$. If there exists such a Frobenius morphism, then the fixed point set $M^{F_M}\in \mod_{\bbF_q}(A^{F_A})$ such that $M=k\otimes_{\bbF_q}M^{F_M}$, and we say $M$ is $F_A$-stable.

{\rm{(4)}} Let $A$ be a $k$-algebra with the Frobenius morphism $F_A:A\rightarrow A$, we define $\mod_k^{F_A}A$ to be the category of $F_A$-stable modules. More precisely, its objects are of the form $(M,F_M)$, where $M\in \mod_kA$ is $F_A$-stable and $F_M:M\rightarrow M$ is the Frobenius morphism; and its morphisms $(M, F_M)\rightarrow (M',F_{M'})$ are morphisms $f:M\rightarrow M'$ of $A$-modules satisfying $fF_M=F_{M'}f$. 
\end{definition}

\begin{theorem}[{\cite[Theorem 3.2]{deng2006frobenius}}]\label{equivalences of categories}
There is an equivalence of categories $$\mod_k^{F_A}A\xrightarrow{\simeq} \mod_{\bbF_q}(A^{F_A})$$
defined by $(M,F_M)\mapsto M^{F_M}$. 
\end{theorem}

For any $(M,F_{M}),(M',F_{M'})\in \mod_k^{F_A}A$, we have $(M,F_{M})\cong (M',F_{M'})$ in $\mod_k^{F_A}A$ if and only if $M^{F_{M}}\cong M'^{F_{M'}}$ in $\mod_{\bbF_q}(A^{F_A})$. Moreover, by \cite[Lemma 3.1]{deng2006frobenius}, we have $M^{F_{M}}\cong M'^{F_{M'}}$ in $\mod_{\bbF_q}(A^{F_A})$ if and only if $M\cong M'$ in $\mod_kA$.

\begin{example}[{\cite[Example 3.5]{deng2006frobenius}}]
Let $Q=(I,H,s,t)$ be a finite quiver without loops, $a$ be an admissible automorphism and $kQ$ be the path algebra. For any trivial path $e_i$ at $i\in I$, we define $a(e_i)=e_{a(i)}$; for any non-trivial path $p=h_1...h_t$, where $h_1,...,h_t\in H$, we define $a(p)=a(h_1)...a(h_t)$. Then there is a Frobenius morphism $F_{Q,a}: kQ\rightarrow kQ$ defined by
\begin{align*}
\sum_s \lambda_s p_s\mapsto \sum_s \lambda_s^q a(p_s)\ \textrm{for any $\lambda_s\in k$ and paths $p_s$.}
\end{align*}
\end{example}

\begin{theorem}[{\cite[Theorem 6.5]{deng2006frobenius}}]\label{relation between algebra and quiver with automorphism}
Let $\Lambda$ be a finite-dimensional hereditary basic algebra over $\bbF_q$, then there exists a finite quiver $Q_\Lambda$ without loops and an admissible automorphism $a$ such that $\Lambda\cong (kQ_\Lambda)^{F_{Q_\Lambda,a}}$.
\end{theorem}

By Theorem \ref{equivalences of categories} and \ref{relation between algebra and quiver with automorphism}, there are equivalences of categories
$$\mod_k^{F_{Q_\Lambda,a}} (kQ_\Lambda)\simeq \mod_{\bbF_q}((kQ_\Lambda)^{F_{Q_\Lambda,a}})\simeq \mod_{\bbF_q}\Lambda.$$

Let $Q=(I,H,s,t)$ be a finite quiver without loops, $a$ be an admissible automorphism and $\rep_k(Q)$ be the category of finite-dimensional representations of $Q$ over $k$. It is well known that there is an equivalence of categories 
$$\Psi:\mod_k (kQ)\xrightarrow{\simeq} \rep_k(Q),$$
see \cite[Theorem 1.6]{Assem-Simson-Skowronski-2006}. More precisely, for any $M\in \mod_k (kQ)$, the representation $\Psi(M)=(V,x)=(\bigoplus_{i\in I}V_i, (x_h:V_{s(h)}\rightarrow V_{t(h)})_{h\in H})$ is given by\\
$\bullet$ $V_i=Me_i$, where $e_i$ is the trivial path at $i$ for any $i\in I$;\\
$\bullet$ $x_h:V_{s(h)}\rightarrow V_{t(h)}$ maps $m\in Me_{s(h)}$ to $mh=me_{s(h)}he_{t(h)}\in Me_{t(h)}$ for any $h\in H$.

\begin{definition}\label{stable representation}
We define $\rep_k^{F_{Q,a}}(Q)$ to be the category of $F_{Q,a}$-stable representations. More precisely, its objects are of the form $(V,x,F_V)$, where $(V,x)\in \rep_k(Q)$ and $F_V:V\rightarrow V$ is a Frobenius map on $V$ satisfying\\ 
{\rm{(a)}} $F_V(V_i)=V_{a(i)}$ for any $i\in I$;\\
{\rm{(b)}} $F_Vx_h=x_{a(h)}F_V:V_{s(h)}\rightarrow V_{t(a(h))}$ for any $h\in H$,\\
and its morphisms $(V,x,F_V)\rightarrow (V',x',F_{V'})$ are morphisms $(f_i)_{i\in I}:(V,x)\rightarrow (V',x')$ of quiver representations satisfying $f_{a(i)}F_V=F_{V'}f_i$ for any $i\in I$. 
\end{definition}

It is routine to check that $\Psi$ induces to an equivalence of categories 
$$\Psi:\mod_k^{F_{Q,a}} (kQ)\xrightarrow{\simeq} \rep_k^{F_{Q,a}}(Q),$$
see \cite[Section 9]{deng2006frobenius}. As a result, for any $(V,x,F_V), (V',x',F_{V'})\in \rep_k^{F_{Q,a}}(Q)$, we have $(V,x)\cong (V',x')$ in $\rep_k(Q)$ if and only if $(V,x,F_V)\cong (V',x',F_{V'})$ in $\rep_k^{F_{Q,a}}(Q)$.

Summarizing above results, we have the following corollary. 

\begin{corollary}\label{corollary relation}
Let $\Lambda$ be a finite-dimensional hereditary basic algebra over $\bbF_q$, then there exists a finite quiver $Q_\Lambda$ without loops and an admissible automorphism $a$ such that
$$\mod_{\bbF_q}\Lambda\simeq \rep_k^{F_{Q_\Lambda,a}}(Q_\Lambda).$$
\end{corollary}

\section{Hall algebra}\label{Hall algebra}

Let $\cA$ be an essentially small, $\bbF_q$-linear, hereditary, finitary, abelian category such that any object has finitely many subobjects, $\textrm{Iso}(\cA)$ be the set of isomorphism classes of objects of $\cA$ and $K(\cA)$ be the Grothendieck group of $\cA$. For any $M\in \cA$, we denote by $[M]\in \textrm{Iso}(\cA)$ its isomorphism class, $\hat{M}\in K(\cA)$ its image.

The Euler form $\langle-,-\rangle$ is a bilinear form on $K(\cA)$ with values in $\bbZ$ induced by $$\langle\hat{M},\hat{N}\rangle=\dim_{\bbF_q}\, \Hom_\cA(M,N)-\dim_{\bbF_q}\, \Ext^1_\cA(M,N)\ \textrm{for any}\ M,N\in \cA.$$ 
The symmetric Euler form $(-,-)$ is given by
$$(\alpha,\beta)=\langle \alpha,\beta\rangle+\langle \beta,\alpha\rangle\ \textrm{for any}\ \alpha,\beta\in K(\cA).$$ 

Let ${v_q}\in \bbC$ be a fixed square root of $q$.

\begin{definition}
{\rm{(1)}} The Hall algebra $\mathcal{H}(\cA)$ for the category $\cA$ is a $\mathbb{C}$-algebra with a basis $\{u_{[M]}|[M]\in {\rm{Iso}}(\cA)\}$ and the multiplication 
$$u_{[M]}*u_{[N]}={v_q}^{\langle\hat{M},\hat{N}\rangle}\sum_{[L]\in {\rm{Iso}}(\cA)}g_{MN}^{L}u_{[L]},$$
where $g_{MN}^{L}$ is the number of submodules $L'\subset L$ satisfying $L/L'\cong M, L'\cong N$.

{\rm{(2)}} The comultiplication on $\mathcal{H}(\cA)$ is defined by
$$\Delta(u_{[L]})=\sum_{[M],[N]\in {\rm{Iso}}(\cA)}{v_q}^{\langle\hat{M},\hat{N}\rangle}g^L_{MN}a_Ma_Na_L^{-1}u_{[M]}\otimes u_{[N]},$$
where $a_{-}=|\Aut_\cA(-)|$ is the order of the automorphism group.
\end{definition}

By Riedtmann-Peng's formula, we have
$$g^L_{MN}=\frac{|\Ext^1_{\cA}(M,N)_L|}{|\Hom_{\cA}(M,N)|}\frac{a_L}{a_Ma_N},$$
where $\Ext^1_{\cA}(M,N)_L\subset \Ext^1_{\cA}(M,N)$ is the subset consisting of extensions whose middle terms are isomorphic to $L$. Hence we have 
$$\Delta(u_{[L]})=\sum_{[M],[N]\in {\rm{Iso}}(\cA)}{v_q}^{\langle\hat{M},\hat{N}\rangle}\frac{|\Ext^1_{\cA}(M,N)_L|}{|\Hom_{\cA}(M,N)|}u_{[M]}\otimes u_{[N]}$$
 
\begin{theorem}[\cite{ringel1990hall,green1995hall}]
{\rm{(1)}} The algebra $\mathcal{H}(\cA)$ is associative with the unit $u_{[0]}$ given by the zero object $0\in \cA$. 

{\rm{(2)}} The coalgebra $\mathcal{H}(\cA)$ is coassociative with the counit $u_{[M]}\mapsto \delta_{M,0}$.
\end{theorem}

\begin{theorem}[Green's formula, \cite{green1995hall}]\label{Green's formula}
For any $M,N,M',N'\in \cA$, we have 
\begin{align*}
&a_Ma_Na_{M'}a_{N'}\sum_{[L]\in {\rm{Iso}}(\cA)}g^L_{MN}g^L_{M'N'}a_L^{-1}\\
=&\sum_{[M_1],[M_2],[N_1],[N_2]\in {\rm{Iso}}(\cA)}q^{-\langle\hat{M_1},\hat{N_2} \rangle}g^M_{M_1M_2}g^N_{N_1N_2}g^{M'}_{M_1N_1}g^{N'}_{M_2N_2}a_{M_1}a_{M_2}a_{N_1}a_{N_2}.
\end{align*}
As a result, the comultiplication map $\Delta:\cH(\cA)\rightarrow \cH(\cA)\otimes \cH(\cA)$ is an algebra homomorphism, where the multiplication on $\cH(\cA)\otimes \cH(\cA)$ is defined by
$$(u_{[M_1]}\otimes u_{[M_2]})*(u_{[N_1]}\otimes u_{[N_2]})={v_q}^{(\hat{M_2},\hat{N_1})}(u_{[M_1]}*u_{[N_1]})\otimes (u_{[M_2]}*u_{[N_2]}).$$
\end{theorem}

Let $\Lambda$ be a finite-dimensional hereditary basic algebra over $\bbF_q$, by Corollary \ref{corollary relation}, there is an equivalence of categories 
$$\mod_{\bbF_q}\Lambda\simeq \rep_k^{F_{Q_\Lambda,a}} (Q_\Lambda).$$
Note that these categories satisfy the assumptions on $\cA$ in the beginning of this section, thus their Hall algebras can be defined. Since the categories are equivalent, we have the bialgebra isomorphism 
$$\cH(\mod_{\bbF_q}\Lambda)\cong \cH(\rep_k^{F_{Q_\Lambda,a}}(Q_\Lambda)).$$

\section{Induction functor and restriction functor}\label{Induction functor and restriction functor}

\subsection{Preliminary}

\subsubsection{Periodic functor}\label{periodic functor}\

In this subsection, we review the definition of the periodic functor and refer \cite[Chapter 11]{lusztig2010introduction} for more details. Let $n$ be a fixed positive integer.

\begin{definition}
{\rm{(1)}} Let $\cC$ be a $\overline{\bbQ}_l$-linear category, a periodic functor on $\cC$ is a linear functor $a^*:\cC\rightarrow \cC$ such that $a^{*n}$ is the identity functor on $\cC$.

{\rm{(2)}} Let $a^*$ be a periodic functor on $\cC$, we define the category $\tilde{\cC}$, also denoted by $\cC^{\ \widetilde{}}$.\\
$\bullet$ Its objects are pairs $(A,\varphi)$, where $A\in \cC$ and $\phi:a^*(A)\rightarrow A$ is an isomorphism in $\cC$ such that the composition
$$A=a^{*n}(A)\xrightarrow{a^{*(n-1)(\varphi)}}a^{*(n-1)}(A)\rightarrow....\rightarrow a^*(A)\xrightarrow{\varphi}A$$
is the identity morphism on $A$.\\
$\bullet$ For any $(A,\varphi),(A',\varphi')\in \tilde{\cC}$, the morphism space
$$\Hom_{\tilde{\cC}}((A,\varphi),(A',\varphi'))=\{f\in \Hom_{\cC}(A,A')|f\varphi=\varphi'(a^*(f))\}.$$ 

{\rm{(3)}} The direct sum of $(A,\varphi),(A',\varphi')\in \tilde{\cC}$ is $(A\oplus A',\varphi\oplus \varphi')$.

{\rm{(4)}} An object $(A,\phi)\in\tilde{\cC}$ is called traceless, if there exists an object $B\in \cC$ and an integer $t\geqslant 2$ dividing $n$ such that $a^{*t}(B)\cong B$, $A\cong B\oplus a^*(B)\oplus...\oplus a^{*(t-1)}(B)$ and $\varphi:a^*(A)\rightarrow A$ corresponds to the isomorphism $a^*(B)\oplus a^{*2}(B)\oplus...\oplus a^{*t}(B)$ taking $a^{*s}(B)$ onto $a^{*s}(B)$ for $1\leqslant s\leqslant t-1$ and the taking $a^{*t}(B)$ onto $B$.
\end{definition} 

\begin{lemma}[{\cite[Section 11.1.3]{lusztig2010introduction}}]\label{split criterion}
Let $(A,\varphi),(A',\varphi'),(A'',\varphi'')$ be objects in $\tilde{\cC}$ and $i':(A',\varphi')\rightarrow (A,\varphi),\ p'':(A,\varphi)\rightarrow (A'',\varphi'')$ be morphisms in $\tilde{\cC}$, if there exists morphisms $i'':A''\rightarrow A,\ p':A\rightarrow A'$ in $\cC$ such that $$p'i'=1_{A'},\ p'i''=0,\ p''i'=0,\ p''i''=1_{A''},\ i'p'+i''p''=1_A,$$
equivalently, $A\cong A'\oplus A''$ in $\cC$, then $(A,\varphi)\cong(A',\varphi')\oplus(A'',\varphi'')$ in $\tilde{\cC}$.
\end{lemma}

Let $\cC'$ be another $\overline{\bbQ}_l$-linear category with a periodic functor $a^*:\cC'\rightarrow \cC'$, and $b:\cC\rightarrow \cC'$ be a linear functor such that $ba^*\cong a^*b:\cC\rightarrow \cC'$, then $b$ induces a linear functor $\tilde{b}:\tilde{\cC}\rightarrow \tilde{\cC'}$ defined by $\tilde{b}(A,\varphi)=(b(A),b(\varphi))$ for any $(A,\varphi)\in \tilde{\cC}$.

\subsubsection{Equivariant mixed semisimple complex and trace map}\label{Equivariant mixed semisimple complex and trace map}\

We refer \cite{Beilinson-Bernstein-Deligne-1982,Kiehl-Rainer-2001,Pramod-2021} for the definitions of the mixed complex on $k$-variety defined over $\bbF_q$ and the trace map, and refer \cite{Bernstein-Lunts-1994,Pramod-2021} for the definition of the equivariant derived category. We fix an isomorphism $\overline{\bbQ}_l\cong \bbC$.

Let $X$ be an algebraic variety over $k$ with a $\bbF_q$-structure and the Frobenius map $F:X\rightarrow F$, and $G$ be a connected algebraic group over $k$ with a $\bbF_q$-structure and the Frobenius map $F:G\rightarrow G$ such that $G$ acts on $X$. We denote by $\cD_G^b(X)$ the $G$-equivariant bounded derived category of constructible $\overline{\bbQ}_l$-sheaves on $X$, $\cD^b_{G,m}(X)$ the subcategory consisting of mixed Weil complexes and $ \cD^{b,ss}_{G,m}(X)$ the subcategory consisting of mixed semisimple complexes.

We define $\overline{\bbQ}_l(\frac{1}{2})$ to be the square root of the $l$-adic Tate sheaf on the point with a $\bbF_q$-structure whose Frobenius map has eigenvalue $(\sqrt{q})^{-1}$. Moreover, for any $n\in \bbZ$, we define $\overline{\bbQ}_l(\frac{n}{2})$ to be $\overline{\bbQ}_l(\frac{1}{2})^{\otimes n}$ if $n\geqslant 0$; or be the dual of $\overline{\bbQ}_l(\frac{1}{2})^{\otimes (-n)}$ if $n<0$.

For any $n\in \bbZ$, we denote by $(\frac{n}{2})$ the functor $-\otimes \overline{\bbQ}_l(\frac{n}{2})$ and denote by $[n]$ the shift functor. For any $G$-equivariant morphism $f:X\rightarrow Y$ which respects the $\bbF_q$-structures, the derived functors of $f^*, f_!$ are still denoted by 
$$f^*:\cD^b_{G,m}(Y)\rightarrow \cD^b_{G,m}(X),\ f_!:\cD^b_{G,m}(X)\rightarrow \cD^b_{G,m}(Y).$$

Let $X^F$ and $G^F$ be the fixed point sets of $X$ and $G$ under their Frobenius maps $F$ respectively, we denote by $\tilde{\cH}_{G^F}(X^F)$ the $\bbC$-vector space of $G^F$-invariant functions on $X^F$. For any $G$-equivariant morphism $f:X\rightarrow Y$ which respects the $\bbF_q$-structures, its restriction to $X^F$ is still denoted by $f:X^F\rightarrow G^F$, and there are $\bbC$-linear maps
\begin{align*}
f^*:\tilde{\cH}_{G^F}(Y^F)&\rightarrow \tilde{\cH}_{G^F}(X^F) &f_!: \tilde{\cH}_{G^F}(X^F)&\rightarrow \tilde{\cH}_{G^F}(Y^F)\\
\varphi&\mapsto (x\mapsto f(\varphi(x))), & \psi&\mapsto(y\mapsto \sum_{x\in f^{-1}(y)}\psi(x)).
\end{align*}

For any mixed Weil complex $A\in \cD^b_{G,m}(X)$ with the Weil structure $\xi:F^*(A)\rightarrow A$ and $x\in X^F, s\in \bbZ$, there is an isomorphism $H^s(\xi)_x:H^s(A)_x\rightarrow H^s(A)_x$ of the stalk at $x$ of the $s$-th cohomology sheaf. Taking the alternative sum of the traces of these isomorphisms, we obtain a value 
$$\chi_A(x)=\sum_{s\in \bbZ}(-1)^s\tr(H^s(\xi)_x)\in\overline{\bbQ}_l\cong \bbC$$
and a function $\chi_A\in \tilde{\cH}_{G^F}(X^F)$. Moreover, $\chi_{-}$ induces a map from the Grothendieck group of $\cD^b_{G,m}(X)$ to $\tilde{\cH}_{G^F}(X^F)$, see \cite[Lemma 5.3.12]{Pramod-2021}, which is called the trace map.

\begin{theorem}[{\cite[Theorem 5.3.13]{Pramod-2021}}]\label{sheaf-function correspondence}
For any $G$-equivariant morphism $f:X\rightarrow Y$ which respects the $\bbF_q$-structures and $A\in \cD^b_{G,m}(X), B\in \cD^b_{G,m}(Y)$, we have
\begin{align*}
&\chi_{A[n]}=(-1)^n\chi_A,\ \chi_{A(\frac{n}{2})}=\sqrt{q}^{-n}\chi_A,\ \chi_{A\boxtimes B}=\chi_A\otimes \chi_B,\\
&\chi_{f^*(B)}=f^*(\chi_B),\ \chi_{f_!(A)}=f_!(\chi_A).
\end{align*}
\end{theorem}

\subsection{Variety for quiver with an automorphism}\label{Variety for quiver with an automorphism}\

Let $Q=(I,H,s,t)$ be a finite quiver without loops, $a$ be an admissible automorphism and $n$ be a fixed positive integer such that $a^n=1$ on $I$ and $H$. 

Let $\bbN I$ be the monoid of $\bbN$-linear combinations of $i\in I$, then the cyclic group $\langle a\rangle $ acts on $\bbN I$ by $a.\sum_{i\in I}\nu_i i=\sum_{i\in I}\nu_{a(i)}i$. We denote by $\bbN I^a\subset \bbN I$ the fixed point set under $a$, that is, the submonoid consisting of $\sum_{i\in I} \nu_{i}i$ satisfying $\nu_{i}=\nu_{a(i)}$ for any $i\in I$. 

For any $\nu\in \bbN I^a$, we fixed a $I$-graded $k$-vector space $\bfV^\nu$ (simply denoted by $\bfV$, if there is no confusion) of dimension vector $\nu$ with\\
$\bullet$ a given Frobenius map $\tilde{F}_{\bfV}:\bfV\rightarrow \bfV$ satisfying $\tilde{F}_{\bfV}(\bfV_i)=\bfV_i$ for any $i\in I$;\\
$\bullet$ a given $k$-linear map $a_{\bfV}:\bfV\rightarrow \bfV$ satisfying $a_{\bfV}\tilde{F}_{\bfV}=\tilde{F}_{\bfV}a_{\bfV}$, $a_{\bfV}(\bfV_i)=\bfV_{a(i)}$ for any $i\in I$, and $a_{\bfV}^s|_{\bfV_i}=1$ for any $i\in I,s\geqslant 1$ whenever $a^s(i)=i$.\\
We denote by $F_{\bfV}=a_{\bfV}\tilde{F}_{\bfV}=\tilde{F}_{\bfV}a_{\bfV}$, then $F_{\bfV}:\bfV\rightarrow \bfV$ is a Frobenius map on $\bfV$ satisfying $F_{\bfV}(\bfV_i)=\bfV_{a(i)}$ for any $i\in I$. 

We define an algebraic group and an affine space
$$\bfG_{\nu}=\prod_{i\in I}\textrm{GL}_k(\bfV_{i}),\ \bfE_{\nu}=\bigoplus_{h\in H}\Hom_k(\bfV_{s(h)},\bfV_{t(h)})$$
such that $\bfG_\nu$ acts on $\bfE_{\nu}$ by $g.x=(g_{t(h)}x_hg_{s(h)}^{-1})_{h\in H}$ for any $g\in \bfG_\nu,x\in \bfE_{\nu}$. 

Then $\bfG_{\nu}$ has a $\bbF_q$-structure with the Frobenius map $F:\bfG_{\nu}\rightarrow \bfG_{\nu}$ satisfying
$$F(g)F_{\bfV}=F_{\bfV}g:\bfV\rightarrow \bfV\ \textrm{for any}\ g\in \bfG_{\nu}$$ 
and $\bfE_{\nu}$ has a $\bbF_q$-structure with the Frobenius map $F:\bfE_{\nu}\rightarrow \bfE_{\nu}$ satisfying
$$(F(x))_hF_{\bfV}=F_{\bfV}x_{a^{-1}(h)}:\bfV_{s(a^{-1}(h))}\rightarrow \bfV_{t(h)}\ \textrm{for any}\ x\in \bfE_{\nu}, h\in H.$$

It is clear that any $x\in \bfE_\nu$ determines a representation $(\bfV,x)\in \rep_k(Q)$, and moreover, $\bfG_\nu.x\mapsto [(\bfV,x)]$ gives a bijection from the set of $\bfG_\nu$-orbits in $\bfE_{\nu}$ to the set of isomorphism classes of objects in $\rep_k(Q)$ of dimension vector $\nu$. The $\bfG_\nu$-action on $\bfE_\nu$ restricts to a $\bfG_\nu^F$-action on $\bfE_\nu^F$, and we have the following result.

\begin{lemma}\label{bijection between orbits and isomorphism classes}
There is a bijection between the set of $\bfG_{\nu}^F$-orbits in $\bfE_{\nu}^F$ and the set of isomorphism classes of objects in $\rep_k^{F_{Q,a}}(Q)$ of dimension vector $\nu$.
\end{lemma}
\begin{proof}
For any $x\in \bfE_{\nu}$, it determines a representation $(\bfV,x)\in \rep_k(Q)$ of dimension vector $\nu$. Recall that the Frobenius map $F_{\bfV}:\bfV\rightarrow \bfV$ satisfies $F_{\bfV}(\bfV_i)=\bfV_{a(i)}$ for any $i\in I$, which is the condition {\rm{(a)}} in Definition \ref{stable representation}. By definition, $x\in \bfE_{\nu}^F$ if and only if $x_hF_{\bfV}=F_{\bfV}x_{a^{-1}(h)}$ for any $h\in H$, which is the condition {\rm{(b)}} in Definition \ref{stable representation}. Thus $x\in \bfE_{\nu}^F$ if and only if $(\bfV,x,F_{\bfV})\in \rep_k^{F_{Q,a}}(Q)$. Moreover, for any $x,x'\in \bfE_{\nu}^F$, we have 
\begin{align*}
&\textrm{$x,x'$ belong to the same $\bfG_{\nu}^F$-orbit}\\
\Leftrightarrow &\textrm{$x,x'$ belong to the same $\bfG_{\nu}$-orbit}
\Leftrightarrow \textrm{$(\bfV,x)\cong (\bfV,x')$ in $\rep_k(Q)$}\\
\Leftrightarrow &\textrm{$(\bfV,x,F_{\bfV})\cong (\bfV,x',F_{\bfV})$ in $\rep_k^{F_{Q,a}}(Q)$}.
\end{align*} 
Hence $\bfG_{\nu}^F.x\mapsto [(\bfV,x,F_{\bfV})]$ defines an injection from the set of $\bfG_{\nu}^F$-orbits in $\bfE_{\nu}^F$ to set of isomorphism classes of objects in $\rep_k^{F_{Q,a}}(Q)$ of dimension vector $\nu$. For any $(V,x',F_V)\in \rep_k^{F_{Q,a}}(Q)$ of dimension vector $\nu$, there exists $(\bfV,x,F_{\bfV})\in \rep_k^{F_{Q,a}}(Q)$ such that $(V,x',F_V)\cong (\bfV,x,F_{\bfV})$, and so the map is also surjective.
\end{proof}

There is an automorphism $a:\bfG_\nu\rightarrow \bfG_\nu$ satisfying
$$a(g)a_{\bfV}=a_{\bfV}g:\bfV\rightarrow \bfV\ \textrm{for any}\ g\in \bfG_\nu,$$
and there is an automorphism $a:\bfE_\nu\rightarrow \bfE_\nu$ satisfying
$$(a(x))_{h}a_{\bfV}=a_{\bfV}x_{a^{-1}(h)}:\bfV_{s(a^{-1}(h))}\rightarrow \bfV_{t(h)} \ \textrm{for any}\ x\in \bfE_\nu,h\in H.$$

Let $\cD^{b,ss}_{\bfG_\nu,m}(\bfE_\nu)$ be the corresponding category in subsection \ref{Equivariant mixed semisimple complex and trace map}, where the Weil structures on its objects are related to $F$. Moreover, we can regard $\tilde{F}_{\bfV}$ as a special case of $F_{\bfV}$ for the quiver $Q$ with the trivial admissible automorphism $a=1$, and denote the corresponding Frobenius maps on $\bfG_\nu,\bfE_\nu$ by $\tilde{F}:\bfG_\nu\rightarrow \bfG_\nu,\tilde{F}:\bfE_\nu\rightarrow \bfE_\nu$ respectively. Let $\cD^{b,ss}_{\bfG_\nu,\tilde{m}}(\bfE_\nu)$ be the corresponding category in \ref{Equivariant mixed semisimple complex and trace map}, where the Weil structures on its objects are related to $\tilde{F}$. 

The automorphisms $a$ on $\bfG_\nu,\bfE_\nu$ satisfy $a^n=1$ such that $$a^*:\cD^{b,ss}_{\bfG_\nu,\tilde{m}}(\bfE_\nu)\rightarrow \cD^{b,ss}_{\bfG_\nu,\tilde{m}}(\bfE_\nu)$$ is a periodic functor on $\cD^{b,ss}_{\bfG_\nu,\tilde{m}}(\bfE_\nu)$, and then the category $(\cD^{b,ss}_{\bfG_\nu,\tilde{m}}(\bfE_\nu))^{\widetilde{}}$ can be defined, see subsection \ref{periodic functor}. The automorphisms $a$ on $\bfG_\nu,\bfE_\nu$ satisfy $F=\tilde{F}a$ such that for any $(A,\varphi)\in(\cD^{b,ss}_{\bfG_\nu,\tilde{m}}(\bfE_\nu))^{\widetilde{}}$, where $A\in \cD^{b,ss}_{\bfG_\nu,\tilde{m}}(\bfE_\nu)$ with the Weil structure $\xi:\tilde{F}^*(A)\rightarrow A$ related to $\tilde{F}$, there is an isomorphism
$$F^*(A)=a^*\tilde{F}^*(A)\xrightarrow{a^*(\xi)}a^*(A)\xrightarrow{\varphi} A$$
which equips $A$ with a Weil structure related to $F$.

\begin{definition}\label{tildecD}
{\rm(1)} We define $\tilde{\cD}_\nu$ to be the subcategory of $(\cD^{b,ss}_{\bfG_\nu,\tilde{m}}(\bfE_\nu))^{\widetilde{}}$ consisting of $(A,\varphi)$, where $A\in \cD^{b,ss}_{\bfG_\nu,\tilde{m}}(\bfE_\nu)$ with the Weil structure $\xi:\tilde{F}^*(A)\rightarrow A$, satisfying $A\in \cD^{b,ss}_{\bfG_\nu,m}(\bfE_\nu)$ with respect to the Weil structure $\varphi a^*(\xi):F^*(A)\rightarrow A$.

{\rm(2)} For any $(A,\varphi)\in \tilde{\cD}_\nu$, where $A\in \cD^{b,ss}_{\bfG_\nu,\tilde{m}}(\bfE_\nu)$ with the Weil structure $\xi:\tilde{F}^*(A)\rightarrow A$, we define $\chi_{(A,\varphi)}$ to be $\chi_A$ by regarding $A\in \cD^{b,ss}_{\bfG_\nu,m}(\bfE_\nu)$ with the Weil structure $\varphi a^*(\xi):F^*(A)\rightarrow A$.
\end{definition}

\begin{lemma}\label{chitraceless}
For any $(A,\varphi)\in \tilde{\cD}_\nu$, if it is traceless, then $\chi_{(A,\varphi)}=0$.
\end{lemma}
\begin{proof}
Let $\xi_A:\tilde{F}^*(A)\rightarrow A$ be the Weil structure of $A\in  \cD^{b,ss}_{\bfG_\nu,\tilde{m}}(\bfE_\nu)$. Since $(A,\varphi)$ is traceless, there exists $B\in \cD^{b,ss}_{\bfG_\nu,\tilde{m}}(\bfE_\nu)$ with the Weil structure $\xi_B:\tilde{F}^*(B)\rightarrow B$ and $t\geqslant 2$ dividing $n$ such that $a^{*t}(B)\cong B, A\cong B\oplus a^*(B)\oplus...\oplus a^{*(t-1)}(B)$ and $\xi_A,\varphi$ corresponds to $
\textrm{diag}(\xi_B, a^*(\xi_B),\cdots,a^{*(t-1)}(\xi_B)),
\begin{pmatrix}
0&1\\
\textrm{Id}_{t-1}  &0
\end{pmatrix}$
respectively. Hence $\tr(H^s(\varphi a^*(\xi_A))_x)=0$ for any $x\in \bfE_\nu^F,s\in \bbZ$, and so $\chi_{(A,\varphi)}=0$.
\end{proof}

For any $\nu',\nu''\in\bbN I^a$, the product $\bfE_{\nu'}\times \bfE_{\nu''}$ can be regarded as a special case of $\bfE_\nu$ for the disjoint union of two copies of the quiver $Q$ with the automorphism $a\times a$, then we have the periodic functor $a^*:\cD^{b,ss}_{\bfG_{\nu'}\times \bfG_{\nu''},\tilde{m}}(\bfE_{\nu'}\times \bfE_{\nu''})\rightarrow \cD^{b,ss}_{\bfG_{\nu'}\times \bfG_{\nu''},\tilde{m}}(\bfE_{\nu'}\times \bfE_{\nu''})$ and the categories $(\cD^{b,ss}_{\bfG_{\nu'}\times \bfG_{\nu''},\tilde{m}}(\bfE_{\nu'}\times \bfE_{\nu''}))^{\widetilde{}},\ \tilde{\cD}_{\nu',\nu''}$.

\subsection{Induction functor}\label{Induction functor}\

For any $\nu,\nu',\nu''\in\bbN I^a$ satisfying $\nu=\nu'+\nu''$, we simply denote by $\bfV=\bfV^{\nu}, \bfV'=\bfV^{\nu'}, \bfV''=\bfV^{\nu''}$ the fixed vector space in subsection \ref{Variety for quiver with an automorphism}.

For any $x\in \bfE_\nu$, a $I$-graded $k$-vector subspace $W\subset \bfV$ is called $x$-stable, if $x_h(W_{s(h)})\subset W_{t(h)}$ for any $h\in H$.

For any $x \in \bfE_\nu$ and $x$-stable, $I$-graded $k$-vector subspace $W\subset \bfV$ of dimension vector $\nu''$, we denote by $x_W:W\rightarrow W$ the restriction and $x_{\bfV/W}:\bfV/W\rightarrow \bfV/W$ the quotient of $x$ respectively. Moreover, for any $I$-graded $k$-linear isomorphisms $\rho_1:\bfV/W\rightarrow \bfV',\ \rho_2:W\rightarrow \bfV''$, we define $\rho_1.x_{\bfV/W}\in \bfE_{\nu'},\ \rho_2.x_W\in \bfE_{\nu'}$ by 
$$(\rho_1.x_{\bfV/W})_h=(\rho_1)_{h''}(x_{\bfV/W})_h(\rho_1)_{h'}^{-1}, (\rho_2.x_W)_h=(\rho_2)_{h''}(x_W)_h(\rho_2)_{h'}^{-1}\ \textrm{for any}\ h\in H.$$

Let $\bfE''$ be the variety of $(x,W)$, where $x \in \bfE_\nu$ and $W\subset \bfV$ is a $x$-stable, $I$-graded $k$-vector subspace of dimension vector $\nu''$, and $\bfE'$ be the variety of $(x,W,\rho_1,\rho_2)$, where $(x,W)\in \bfE''$ and $\rho_1:\bfV/W\rightarrow \bfV', \rho_2:W\rightarrow \bfV''$ are $I$-graded $k$-linear isomorphisms, such that $\bfG_{\nu}$ acts on $\bfE''$ by
$$g.(x,W)=(g.x,g(W))\ \textrm{for any}\ g\in \bfG_{\nu}, (x,W)\in \bfE'',$$
and $\bfG_{\nu'}\times \bfG_{\nu''}\times \bfG_{\nu}$ acts on $\bfE'$ by 
$$(g',g'',g).(x,W,\rho_1,\rho_2)=(g.x,g(W),g'\rho_1g^{-1},g''\rho_2g^{-1})$$
for any $(g',g'',g)\in \bfG_{\nu'}\times \bfG_{\nu''}\times \bfG_{\nu}, (x,W,\rho_1,\rho_2)\in \bfE'$.

Then $\bfE''$ has a $\bbF_q$-structure with the Frobenius map $F:\bfE''\rightarrow \bfE''$ defined by
$$F(x,W)=(F(x),F_{\bfV}(W))\ \textrm{for any}\ (x,W)\in \bfE'',$$
and $\bfE'$ has a $\bbF_q$-structure with the Frobenius map $F:\bfE'\rightarrow \bfE'$ defined by
$$F(x,W,\rho_1,\rho_2)=(F(x),F_{\bfV}(W),F(\rho_1),F(\rho_2))\ \textrm{for any}\ (x,W,\rho_1,\rho_2)\in \bfE',$$
where $F(\rho_1):\bfV/F_{\bfV}(W)\rightarrow \bfV',\ F(\rho_2):F_{\bfV}(W)\rightarrow \bfV''$ are defined by
\begin{align*}
&F(\rho_1)(F_{\bfV}(m)+F_{\bfV}(W))=F_{\bfV'}(\rho_1(m+W))\ \textrm{for any}\ m\in \bfV,\\
&F(\rho_2)(F_{\bfV}(m'))=F_{\bfV''}(\rho_2(m'))\ \textrm{for any}\ m'\in W.
\end{align*}

There is an automorphism $a:\bfE''\rightarrow \bfE''$ defined by
$$a(x,W)=(a(x),a_{\bfV}(W))\ \textrm{for any}\ (x,W)\in \bfE'',$$
and there is an automorphism $a:\bfE'\rightarrow \bfE'$ defined by
$$a(x,W,\rho_1,\rho_2)=(a(x),a_{\bfV}(W),a(\rho_1),a(\rho_2))\ \textrm{for any}\ (x,W,\rho_1,\rho_2)\in \bfE',$$
where $a(\rho_1):\bfV/a_{\bfV}(W)\rightarrow \bfV',a(\rho_2):W\rightarrow \bfV''$ are defined by
\begin{align*}
&a(\rho_1)(a_{\bfV}(m)+a_{\bfV}(W))=a_{\bfV'}(\rho_1(m+W))\ \textrm{for any}\ m\in \bfV,\\
&a(\rho_2)(a_{\bfV}(m'))=a_{\bfV''}(\rho_2(m'))\ \textrm{for any}\ m'\in W.
\end{align*}

Let $\cD^{b,ss}_{\bfG_{\nu'}\times \bfG_{\nu''}\times \bfG_{\nu},m}(\bfE'),\ \cD^{b,ss}_{\bfG_{\nu},m}(\bfE'')$ be the corresponding categories in subsection \ref{Equivariant mixed semisimple complex and trace map}, where the Weil structures on their objects are related to $F$. Moreover, we can regard $\tilde{F}_{\bfV}$ as a special case of $F_{\bfV}$ for the quiver $Q$ with the trivial admissible automorphism $a=1$, and denote the corresponding Frobenius maps on $\bfE',\bfE''$ by $\tilde{F}:\bfE'\rightarrow \bfE',\ \tilde{F}:\bfE''\rightarrow \bfE''$ respectively. Let $\cD^{b,ss}_{\bfG_{\nu'}\times \bfG_{\nu''}\times \bfG_{\nu},\tilde{m}}(\bfE'),\ \cD^{b,ss}_{\bfG_{\nu},\tilde{m}}(\bfE'')$ be the corresponding category in \ref{Equivariant mixed semisimple complex and trace map}, where the Weil structures on their objects are related to $\tilde{F}$.

The automorphism $a$ on $\bfE',\bfE''$ satisfy $a^n=1$ such that
\begin{align*}
&a^*:\cD^{b,ss}_{\bfG_{\nu'}\times \bfG_{\nu''}\times \bfG_{\nu},\tilde{m}}(\bfE')\rightarrow \cD^{b,ss}_{\bfG_{\nu'}\times \bfG_{\nu''}\times \bfG_{\nu},\tilde{m}}(\bfE'),\\
&a^*:\cD^{b,ss}_{\bfG_{\nu},\tilde{m}}(\bfE'')\rightarrow \cD^{b,ss}_{\bfG_{\nu},\tilde{m}}(\bfE''),
\end{align*}
are periodic functors, and then the categories $(\cD^{b,ss}_{\bfG_{\nu'}\times \bfG_{\nu''}\times \bfG_{\nu},\tilde{m}}(\bfE'))^{\widetilde{}}, (\cD^{b,ss}_{\bfG_{\nu},\tilde{m}}(\bfE''))^{\widetilde{}}$ can be defined. We define the subcategories $\tilde{\cD}',\tilde{\cD}''$ in  a similar way to Definition \ref{tildecD}.

Consider the following morphisms
\begin{diagram}[midshaft,size=1.5em]
\bfE_{\nu'}\times \bfE_{\nu''} &\lTo^{p_1} &\bfE' &\rTo^{p_2} &\bfE'' &\rTo^{p_3} &\bfE_{\nu}\\
(\rho_1.x_{\bfV/W}, \rho_2.x_W) &\lMapsto &(x,W,\rho_1,\rho_2) &\rMapsto &(x,W) &\rMapsto &x,
\end{diagram}
which respect the $\bbF_q$-structures defined by $\tilde{F}$ and $F$. \\
$\bullet$ The morphism $p_1$ is smooth with connected fibers and $\bfG_{\nu'}\times \bfG_{\nu''}\times \bfG_{\nu}$-equivariant with respect to the trivial $\bfG_\nu$-action on $\bfE_{\nu'}\times \bfE_{\nu''}$. By \cite[Proposition 3.6.1]{Pramod-2021}, we have $p_1^*:\cD^{b,ss}_{\bfG_{\nu'}\times \bfG_{\nu''},\tilde{m}}(\bfE_{\nu'}\times \bfE_{\nu''})\rightarrow \cD^{b,ss}_{\bfG_{\nu'}\times \bfG_{\nu''}\times \bfG_{\nu},\tilde{m}}(\bfE')$. By $p_1a=ap_1, a^*p_1^*=p_1^*a^*$, we have $\tilde{p_1^*}:(\cD^{b,ss}_{\bfG_{\nu'}\times \bfG_{\nu''},\tilde{m}}(\bfE_{\nu'}\times \bfE_{\nu''}))^{\widetilde{}}\rightarrow (\cD^{b,ss}_{\bfG_{\nu'}\times \bfG_{\nu''}\times \bfG_{\nu},\tilde{m}}(\bfE'))^{\widetilde{}}$ and $\tilde{p_1^*}:\tilde{\cD}_{\nu',\nu''}\rightarrow \tilde{\cD}'$.\\
$\bullet$ The morphism $p_2$ is a $\bfG_{\nu'}\times \bfG_{\nu''}$-principal bundle and $\bfG_{\nu}$-equivariant. By \cite[Section 8.1.8]{lusztig2010introduction}, $p_2^*:\cD^{b,ss}_{\bfG_{\nu'}\times \bfG_{\nu''}\times \bfG_{\nu},\tilde{m}}(\bfE')\rightarrow \cD^{b,ss}_{\bfG_{\nu},\tilde{m}}(\bfE'')$ is an equivalence of categories with the quasi-inverse denoted by ${p_2}_\flat$. By $p_2a=ap_2, a^*p_2^*=p_2^*a^*,a^*{p_2}_\flat\cong {p_2}_\flat a^*$, we have $\tilde{{p_2}_\flat}:(\cD^{b,ss}_{\bfG_{\nu'}\times \bfG_{\nu''}\times \bfG_{\nu},\tilde{m}}(\bfE'))^{\widetilde{}}\rightarrow(\cD^{b,ss}_{\bfG_{\nu},\tilde{m}}(\bfE''))^{\widetilde{}}$ and $\tilde{{p_2}_\flat}:\tilde{\cD}'\rightarrow \tilde{\cD}''$.\\
$\bullet$ The morphism $p_3$ is proper and $\bfG_{\nu}$-equivariant. By \cite[Theorem 3.9.2]{Pramod-2021}, we have ${p_3}_!:\cD^{b,ss}_{\bfG_{\nu},\tilde{m}}(\bfE'')\rightarrow \cD^{b,ss}_{\bfG_\nu,\tilde{m}}(\bfE_\nu)$. There is a cartesian diagram
\begin{diagram}[midshaft,size=2em]
\bfE'' &\rTo^{p_3} &\bfE_\nu\\
\dTo^{a} &\square &\dTo_{a}\\
\bfE'' &\rTo^{p_3} &\bfE_\nu
\end{diagram}
such that $a^*{p_3}_!\cong {p_3}_!a^*$ by base change. We have $\tilde{{p_3}_!}\!:\!(\cD^{b,ss}_{\bfG_{\nu},\tilde{m}}(\bfE''))^{\widetilde{}}\rightarrow (\cD^{b,ss}_{\bfG_\nu,\tilde{m}}(\bfE_\nu))^{\widetilde{}}$ and $\tilde{{p_3}_!}:\tilde{\cD}''\rightarrow \tilde{\cD}_\nu$.

\begin{definition}
The induction functor $\Ind^{\nu}_{\nu',\nu''}:\tilde{\cD}_{\nu'}\boxtimes \tilde{\cD}_{\nu''}\rightarrow \tilde{\cD}_\nu$ is defined by
$$\Ind^{\nu}_{\nu',\nu''}((A,\varphi)\boxtimes (B,\psi))=\tilde{{p_3}_!}\tilde{{p_2}_\flat}\tilde{p_1^*}((A,\varphi)\boxtimes (B,\psi))[d_1-d_2](\frac{d_1-d_2}{2})$$
for any $(A,\varphi)\in \tilde{\cD}_{\nu'}, (B,\psi)\in \tilde{\cD}_{\nu''}$, where $d_1$ and $d_2$ are the dimensions of the fibers of $p_1$ and $p_2$ respectively, and 
$$d_1-d_2=\sum_{i\in I}\nu'_i\nu''_i+\sum_{h\in H}\nu'_{s(h)}\nu''_{t(h)}.$$
\end{definition}

\subsection{Restriction functor}\label{Restriction functor}\

For any $\nu,\nu',\nu''\in\bbN I^a$ satisfying $\nu=\nu'+\nu''$, we simply denote by $\bfV=\bfV^{\nu}, \bfV'=\bfV^{\nu'}, \bfV''=\bfV^{\nu''}$ the fixed vector space in subsection \ref{Variety for quiver with an automorphism}. In this subsection, we fix an $a_{\bfV}$-stable, $\tilde{F}_{\bfV}$-stable, $I$-graded $k$-vector subspace $\bfW\subset \bfV$ of dimension vector $\nu''$ and two $I$-graded $k$-linear isomorphisms $\rho_1:\bfV/\bfW\rightarrow \bfV'$ and $\rho_2:\bfW\rightarrow \bfV''$. 

Let $\bfQ\subset \bfG_\nu$ be the stabilizer of $\bfW\subset \bfV$ and $\bfU\subset \bfQ$ be its unipotent radical, then $\bfQ$ acts on $\bfE_{\nu'}\times \bfE_{\nu''}$ through the quotient $\bfQ/\bfU\cong \bfG_{\nu'}\times \bfG_{\nu''}$. In particular, $\bfU$ acts trivially. Both $\bfQ$ and $\bfU$ are stable under the Frobenius map $F:\bfG_{\nu}\rightarrow \bfG_{\nu}$ and the automorphism $a:\bfG_{\nu}\rightarrow \bfG_{\nu}$, we still denote their restrictions to $\bfQ$ and $\bfU$ by the same notations.

Let $\bfF$ be the closed subvariety of $\bfE_\nu$ consisting of $x\in \bfE_\nu$ such that $\bfW$ is $x$-stable, then $\bfQ$ acts on $\bfF$, and $\bfF$ is stable under the Frobenius map $F:\bfE_{\nu}\rightarrow \bfE_{\nu}$ and the automorphism $a:\bfE_{\nu}\rightarrow \bfE_{\nu}$. We still denote by $F:\bfF\rightarrow \bfF,\ a:\bfF\rightarrow \bfF$ the restrictions of $F, a$ respectively.

Let $\cD^{b,ss}_{\bfQ,m}(\bfE_\nu),\ \cD^{b,ss}_{\bfQ,m}(\bfF),\ \cD^{b,ss}_{\bfQ,m}(\bfE_{\nu'}\times \bfE_{\nu''})$ be the corresponding categories in subsection \ref{Equivariant mixed semisimple complex and trace map}, where the Weil structures on their objects are related to $F$. Moreover, we can regard $\tilde{F}_{\bfV}$ as a special case of $F_{\bfV}$ for the quiver $Q$ with the trivial admissible automorphism $a=1$, and denote the corresponding Frobenius maps on $\bfQ,\bfF$ by $\tilde{F}:\bfQ\rightarrow \bfQ,\ \tilde{F}:\bfF\rightarrow \bfF$ respectively. Let $\cD^{b,ss}_{\bfQ,\tilde{m}}(\bfE_\nu),\ \cD^{b,ss}_{\bfQ,\tilde{m}}(\bfF),\ \cD^{b,ss}_{\bfQ,\tilde{m}}(\bfE_{\nu'}\times \bfE_{\nu''})$ be the corresponding category in \ref{Equivariant mixed semisimple complex and trace map}, where the Weil structures on their objects are related to $\tilde{F}$.
    
Consider the following morphisms 
\begin{diagram}[midshaft,size=1.5em]
\bfE_{\nu'}\times \bfE_{\nu''} &\lTo^{\kappa} &\bfF &\rTo^{\iota} &\bfE_{\nu}\\
(\rho_1.x_{\bfV/\bfW}, \rho_2.x_{\bfW}) &\lMapsto &x &\rMapsto &x 
\end{diagram}
which respect the $\bbF_q$-structures defined by $\tilde{F}$ and $F$.\\
$\bullet$ The morphism $\iota$ is the closed embedding and $\bfQ$-equivariant. By $\iota a=a\iota, a^*\iota^*=\iota^*a^*$, we have $\tilde{\iota^*}:(\cD^b_{\bfQ,\tilde{m}}(\bfE_{\nu}))^{\widetilde{}}\rightarrow (\cD^b_{\bfQ,\tilde{m}}(\bfF))^{\widetilde{}}$.\\
$\bullet$ The morphism $\kappa$ is a vector bundle of rank $\sum_{h\in H}\nu'_{s(h)}\nu''_{t(h)}$ and $\bfQ$-equivariant. There is a cartesian diagram
\begin{diagram}[midshaft,size=2em]
\bfF &\rTo^{\kappa} &\bfE_{\nu'}\times \bfE_{\nu''}\\
\dTo^a &\square &\dTo_a\\
\bfF &\rTo^{\kappa} &\bfE_{\nu'}\times \bfE_{\nu''},
\end{diagram}
such that $a^*\kappa_!\cong\kappa_!a^*$. We have $\tilde{\kappa_!}:(\cD^b_{\bfQ,\tilde{m}}(\bfF))^{\widetilde{}}\rightarrow (\cD^b_{\bfQ,\tilde{m}}(\bfE_{\nu'}\times \bfE_{\nu''}))^{\widetilde{}}$.

For any $(A,\varphi)\in \tilde{\cD}_{\nu}$, the object $A\in \cD^{b,ss}_{\bfG_\nu,\tilde{m}}(\bfE_{\nu})$ can be viewed as an object in $\cD^{b,ss}_{\bfQ,\tilde{m}}(\bfE_{\nu})$ via the forgetful functor $\cD^{b,ss}_{\bfG_\nu,\tilde{m}}(\bfE_{\nu})\rightarrow \cD^{b,ss}_{\bfQ,\tilde{m}}(\bfE_{\nu})$. By \cite[Proposition 2.10]{fang2023parity}, $\kappa_!\iota^*:\cD^b_{\bfQ,\tilde{m}}(\bfE_{\nu})\rightarrow \cD^b_{\bfQ,\tilde{m}}(\bfE_{\nu'}\times \bfE_{\nu''})$ is hyperbolic localization functor in the sense of \cite{braden-2003} and $\kappa_!\iota^*(A)\in\cD^{b,ss}_{\bfQ,\tilde{m}}(\bfE_{\nu'}\times \bfE_{\nu''})$. By \cite[Theorem 6.6.16]{Pramod-2021}, there is an equivalence of categories $\cD^{b,ss}_{\bfQ,\tilde{m}}(\bfE_{\nu'}\times \bfE_{\nu''})\simeq \cD^{b,ss}_{\bfG_{\nu'}\times \bfG_{\nu''},\tilde{m}}(\bfE_{\nu'}\times \bfE_{\nu''})$ such that $\kappa_!\iota^*(A)$ can be viewed as an object in $\cD^{b,ss}_{\bfG_{\nu'}\times \bfG_{\nu''},\tilde{m}}(\bfE_{\nu'}\times \bfE_{\nu''})$. By Lemma \ref{split criterion}, we have $\tilde{\kappa_!}\tilde{\iota^*}(A,\varphi)\in (\cD^{b,ss}_{\bfG_{\nu'}\times \bfG_{\nu''},\tilde{m}}(\bfE_{\nu'}\times \bfE_{\nu''}))^{\widetilde{}}$ and moreover, $\tilde{\kappa_!}\tilde{\iota^*}(A,\varphi)\in\tilde{\cD}_{\nu',\nu''}$.

\begin{definition}
The restriction functor $\Res^\nu_{\nu',\nu''}:\tilde{\cD}_{\nu}\rightarrow \tilde{\cD}_{\nu',\nu''}$ is defined by
$$\Res^\nu_{\nu',\nu''}(A,\varphi)=\tilde{\kappa_!}\tilde{\iota^*}(A,\varphi)[-\langle\nu',\nu''\rangle](-\frac{\langle\nu',\nu''\rangle}{2}),$$
for any $(A,\varphi)\in \tilde{\cD}_{\nu}$, where $\langle-,-\rangle:\bbZ I\times \bbZ I\rightarrow \bbZ$ is a bilinear form defined by
$$\langle \nu',\nu''\rangle=\sum_{i\in I}\nu'_i\nu''_i-\sum_{h\in H}\nu'_{s(h)}\nu''_{t(h)}.$$
\end{definition}
We remark that $\langle-,-\rangle$ coincides with the Euler form for $\rep_k^{F_Q,a}(Q)$, see section \ref{Hall algebra}. More precisely, taking the dimension vectors induce an isomorphism from the Grothendieck group $K(\rep_k^{F_Q,a}(Q))$ to $\bbZ I$ such that 
$$\langle\hat{M},\hat{M'}\rangle=\langle \nu',\nu''\rangle$$ 
for any $M,M'\in \rep_k^{F_Q,a}(Q)$ of dimension vector $\nu,\nu'$ respectively.

\subsection{Hall algebra via function}\label{Hall algebra via function}\

In this subsection, we use the same notations as them in subsections \ref{Variety for quiver with an automorphism}, \ref{Induction functor}, \ref{Restriction functor}. Let ${v_q}\in \bbC$ be the fixed square root of $q$ which is the same as it in section \ref{Hall algebra}.

For any $\nu\in \bbN I^a$, $\tilde{\cH}_{\bfG_\nu^F}(\bfE_\nu^F)$ is the $\bbC$-vector space of $\bfG_\nu^F$-invariant functions on $\bfE_\nu^F$ which is finite-dimensional with a basis consisting of characteristic functions of $\bfG_\nu^F$-orbits in $\bfE_\nu^F$. For any $\nu',\nu''\in \bbN I^a$, there is an isomorphism of $\bbC$-vector space
\begin{align*}
\tilde{\cH}_{\bfG_{\nu'}^F}(\bfE_{\nu'}^F)\otimes \tilde{\cH}_{\bfG_{\nu''}^F}(\bfE_{\nu''}^F)&\rightarrow \tilde{\cH}_{\bfG_{\nu'}^F\times \bfG_{\nu''}^F}(\bfE_{\nu'}^F\times \bfE_{\nu''}^F)\\
f\otimes g&\mapsto ((x',x'')\mapsto f(x')g(x'')).
\end{align*}

For any $\nu,\nu',\nu''\in \bbN I^a$ satisfying $\nu=\nu'+\nu''$, the varieties and morphisms 
 \begin{diagram}[midshaft,size=1.5em]
\bfE_{\nu'}\times \bfE_{\nu''} &\lTo^{p_1} &\bfE' &\rTo^{p_2} &\bfE'' &\rTo^{p_3} &\bfE_{\nu}, & &\bfE_{\nu'}\times \bfE_{\nu''} &\lTo^{\kappa} &\bfF &\rTo^{\iota} &\bfE_{\nu}
\end{diagram}
are defined over $\bbF_q$, see subsections \ref{Induction functor} and \ref{Restriction functor}. Taking the fixed point sets under their Frobenius maps $F$, we obtain 
 \begin{diagram}[midshaft,size=1.5em]
\bfE_{\nu'}^F\times \bfE_{\nu''}^F &\lTo^{p_1} &\bfE'^F &\rTo^{p_2} &\bfE''^F &\rTo^{p_3} &\bfE_{\nu}^F, & &\bfE_{\nu'}^F\times \bfE_{\nu''}^F &\lTo^{\kappa} &\bfF^F &\rTo^{\iota} &\bfE_{\nu}^F,
\end{diagram}
where $p_1,p_2,p_3,\kappa,\iota$ are the restrictions denoted by the same notations.

For any $f\in \tilde{\cH}_{\bfG_{\nu'}^F}(\bfE_{\nu'}^F),\ g\in \tilde{\cH}_{\bfG_{\nu''}^F}(\bfE_{\nu''}^F)$, we have 
\begin{align*}
f\otimes g\in \tilde{\cH}_{\bfG_{\nu'}^F}(\bfE_{\nu'}^F)\otimes \tilde{\cH}_{\bfG_{\nu''}^F}(\bfE_{\nu''}^F)\cong \tilde{\cH}_{\bfG_{\nu'}^F\times \bfG_{\nu''}^F}(\bfE_{\nu'}^F\!\times\! \bfE_{\nu''}^F)
\end{align*}
and $p_1^*(f\otimes g)\in \tilde{\cH}_{\bfG_{\nu'}^F\times \bfG_{\nu''}^F\times \bfG_{\nu}^F}(\bfE')$ which is $\bfG_{\nu'}^F\times \bfG_{\nu''}^F$-invariant. Since $p_2:\bfE'^F\rightarrow \bfE''^F$ is a $\bfG_{\nu'}^F\times \bfG_{\nu''}^F$-principal bundle, there exists a unique $h\in \tilde{\cH}_{\bfG_{\nu}^F}(\bfE')$ such that $p_2^*(h)=p_1^*(f\otimes g)$. In fact, $h=|\bfG_{\nu'}^F\times \bfG_{\nu''}^F|^{-1}{p_2}_!p_1^*(f\otimes g)$. We can form ${p_3}_!(h)\in \tilde{\cH}_{\bfG_\nu^F}(\bfE_\nu^F)$. Similarly, for any $f\in \tilde{\cH}_{\bfG_\nu^F}(\bfE_\nu^F)$, it can be viewed as an element in $\tilde{\cH}_{\bfQ^F}(\bfE_\nu^F)$, then we can form $\kappa_!\iota^*(f)\in \tilde{\cH}_{\bfQ^F}(\bfE_{\nu'}^F\times \bfE_{\nu''}^F)$. Since the unipotent radical $\bfU^F$ acts trivially on $\bfE_{\nu'}^F\times \bfE_{\nu''}^F$, there are isomorphisms 
$$\tilde{\cH}_{\bfQ^F}(\bfE_{\nu'}^F\times \bfE_{\nu''}^F)\cong \tilde{\cH}_{\bfG_{\nu'}^F\times \bfG_{\nu''}^F}(\bfE_{\nu'}^F\times \bfE_{\nu''}^F)\cong \tilde{\cH}_{\bfG_{\nu'}^F}(\bfE_{\nu'}^F)\otimes \tilde{\cH}_{\bfG_{\nu''}^F}(\bfE_{\nu''}^F)$$
such that $\kappa_!\iota^*(f)$ can be viewed as an element in $\tilde{\cH}_{\bfG_{\nu'}^F}(\bfE_{\nu'}^F)\otimes \tilde{\cH}_{\bfG_{\nu''}^F}(\bfE_{\nu''}^F)$.

\begin{definition}
{\rm{(1)}} The induction $\ind^\nu_{\nu',\nu''}:\tilde{\cH}_{\bfG_{\nu'}^F}(\bfE_{\nu'}^F)\otimes \tilde{\cH}_{\bfG_{\nu''}^F}(\bfE_{\nu''}^F)\rightarrow \tilde{\cH}_{\bfG_\nu^F}(\bfE_\nu^F)$
is defined by 
$$\ind^\nu_{\nu',\nu''}(f\otimes g)=\frac{{v_q}^{-\sum_{i\in I}\nu'_i\nu''_i-\sum_{h\in H}\nu'_{s(h)}\nu''_{t(h)}}}{|\bfG_{\nu'}^F\times \bfG_{\nu''}^F|}{p_3}_!{p_2}_!p_1^*(f\otimes g).$$

{\rm{(2)}} The restriction $\res^\nu_{\nu',\nu''}:\tilde{\cH}_{\bfG_\nu^F}(\bfE_\nu^F)\rightarrow \tilde{\cH}_{\bfG_{\nu'}^F}(\bfE_{\nu'}^F)\otimes \tilde{\cH}_{\bfG_{\nu''}^F}(\bfE_{\nu''}^F)$
is defined by
$$\res^\nu_{\nu',\nu''}(f)={v_q}^{\sum_{i\in I}\nu'_i\nu''_i-\sum_{h\in H}\nu'_{s(h)}\nu''_{t(h)}}\kappa_!\iota^*(f).$$
\end{definition}

We denote by $\cA=\rep_k^{F_{Q,a}}(Q)$ and define $\tilde{\cH}(\cA)=\bigoplus_{\nu\in \bbN I^a}\tilde{\cH}_{\bfG_\nu^F}(\bfE_\nu^F)$ such that all $\ind^\nu_{\nu',\nu''}$ define a multiplication $\tilde{*}$ and all $\res^{\nu}_{\nu',\nu''}$ define a multiplication $\tilde{\Delta}$ on it. 

By Lemma \ref{bijection between orbits and isomorphism classes}, there is a bijection between the set of $\bfG_\nu^F$-orbits in $\bfE^F_\nu$ and the set of isomorphism classes of objects in $\cA$ of dimension vector $\nu$ given by $\bfG^F_\nu.x\mapsto [(\bfV,x,F_{\bfV})]$. For any isomorphism class $[M]$ of objects in $\cA$ of dimension vector $\nu$, we denote by $\cO_M$ the corresponding $\bfG_\nu^F$-orbits in $\bfE^F_\nu$ and $1_{\cO_M}$ the corresponding characteristic function. We define a $\bbC$-linear isomorphism $$\Phi:\tilde{\cH}(\cA)\rightarrow \cH(\cA)$$ by 
$1_{\cO_M}\mapsto {v_q}^{\sum_{i\in I}\nu_i^2}u_{[M]}$.

\begin{proposition}\label{preserve structure constant}
The $\bbC$-linear isomorphism $\Phi:\tilde{\cH}(\cA)\rightarrow \cH(\cA)$ preserves the multiplication and the comultiplication. 
\end{proposition}
\begin{proof}
$\bullet$ For any $M,N\in \cA$ of dimension vectors $\nu',\nu''$ respectively, we have 
\begin{align*}
&\Phi(1_{\cO_M})*\Phi(1_{\cO_N})={v_q}^{\sum_{i\in I}{\nu'_i}^2+\sum_{i\in I}{\nu''_i}^2}{v_q}^{\langle \hat{M},\hat{N}\rangle}\sum_{[L]}g^L_{MN}u_{[L]}\\
=&{v_q}^{-\sum_{i\in I}\nu'_i\nu''_i-\sum_{h\in H}\nu'_{s(h)}\nu''_{t(h)}}\sum_{[L]}g^L_{MN}\Phi(1_{\cO_L}),
\end{align*}
where we use the fact that $L\in \cA$ is of dimension vector $\nu=\nu'+\nu''$ whenever $g^L_{MN}\not=0$. For any $x\in \bfE_\nu^F$ satisfying $\nu=\nu'+\nu''$, let $L=(\bfV,x,F_{\bfV})$, then 
\begin{align*}
&({p_3}_!{p_2}_!p_1^*(1_{\cO_M}\otimes 1_{\cO_N}))(x)
=\sum_{(x,W,\rho_1,\rho_2)\in \bfE'^F}(p_1^*(1_{\cO_M}\otimes 1_{\cO_N}))(x,W,\rho_1,\rho_2)\\
=&\sum_{(x,W,\rho_1,\rho_2)\in \bfE'^F}1_{\cO_M}(\rho_1.x_{\bfV/W})1_{\cO_N}(\rho_2.x_{W})=|\tilde{G}^L_{MN}|,
\end{align*} 
where $\tilde{G}^L_{MN}$ is the set of $(x,W,\rho_1,\rho_2)\in \bfE'^F$ satisfying $\rho_1.x_{\bfV/W}\in \cO_M, \rho_2.x_{W}\in \cO_N$, equivalently, $(\bfV',\rho_1.x_{\bfV/W},F_{\bfV'})\cong M, (\bfV'',\rho_2.x_{W},F_{\bfV''})\cong N$ in $\cA$. Let $G^L_{MN}$ be the set of subobjects $L'$ of $L$ satisfying $L/L'\cong M, L'\cong N$ in $\cA$, then there is a map $\tilde{G}^L_{MN}\rightarrow G^L_{MN}$ defined by$(x,W,\rho_1,\rho_2)\mapsto (W, x_W, F_{\bfV}|_W)$
which is a surjection with fibers isomorphic to $\bfG_{\nu'}^F\times \bfG_{\nu''}^F$. Hence we have
\begin{align*}
&|\tilde{G}^L_{MN}|=|\bfG_{\nu'}^F\times \bfG_{\nu''}^F||G^L_{MN}|=|\bfG_{\nu'}^F\times \bfG_{\nu''}^F|g^L_{MN},\\
&(\ind^{\nu}_{\nu',\nu''}(1_{\cO_M}\otimes 1_{\cO_N}))(x)={v_q}^{-\sum_{i\in I}\nu'_i\nu''_i-\sum_{h\in H}\nu'_{s(h)}\nu''_{t(h)}}g^L_{MN}
\end{align*}
and so $$\ind^{\nu}_{\nu',\nu''}(1_{\cO_M}\otimes 1_{\cO_N})={v_q}^{-\sum_{i\in I}\nu'_i\nu''_i-\sum_{h\in H}\nu'_{s(h)}\nu''_{t(h)}}\sum_{[L]}g^L_{MN}1_{\cO_L}.$$
Therefore, $\Phi$ preserves the multiplication.\\
$\bullet$ For any $L\in \cA$ of dimension vector $\nu$, we have
\begin{align*}
&\Delta(\Phi(1_{\cO_L}))=\sum_{[M],[N]}{v_q}^{\sum_{i\in I}\nu_i^2}{v_q}^{\langle\hat{M},\hat{N}\rangle}\frac{|\Ext^1_{\cA}(M,N)_L|}{|\Hom_{\cA}(M,N)|}u_{[M]}\otimes u_{[N]}\\
=&\sum_{\nu',\nu''}\sum_{\cO_M\subset\bfE^F_{\nu'},\cO_N\subset \bfE_{\nu''}^F} {{v_q}}^{3\sum_{i\in I}\nu'_i\nu''_i-\sum_{h\in H}\nu'_{s(h)}\nu''_{t(h)}}\frac{|\Ext^1_{\cA}(M,N)_L|}{|\Hom_{\cA}(M,N)|}\Phi(1_{\cO_M})\otimes \Phi(1_{\cO_N}),
\end{align*}
where we use the fact that $\nu=\nu'+\nu''$ whenever $|\Ext^1_{\cA}(M,N)_L|\not=0$ for any $M,N\in \cA$ of dimension vectors $\nu',\nu''$ respectively. For any $x'\in \bfE_{\nu'}^F$ and $x''\in \bfE_{\nu''}^F$ satisfying $\nu=\nu'+\nu''$, let $M=(\bfV',x',F_{\bfV'})$ and $N=(\bfV'',x'',F_{\bfV''})$, then 
\begin{align*}
(\kappa_!\iota^*(1_{\cO_L}))(x',x'')=&\sum_{x\in \kappa^{-1}(x',x'')}(\iota^*(1_{\cO_L}))(x)\\
=&\sum_{x\in \kappa^{-1}(x',x'')}1_{\cO_L}(x)=|\kappa^{-1}(x',x'')\cap \cO_L|.
\end{align*}
We identify $\bfV\cong \bfV/\bfW\oplus \bfW\cong \bfV'\oplus \bfV''$, where $\bfW$ is the fixed $I$-graded subspace appearing in the definition of $\Res^\nu_{\nu',\nu''}$, and write elements $x\in \bfE_{\nu}^F$ into the matrix form, then the fiber of $\kappa:\bfF^F\rightarrow \bfE_{\nu'}^F\times \bfE_{\nu''}^F$ at $(x',x'')$ is identified with 
$$\{\begin{pmatrix}
x' &0\\
y &x''
\end{pmatrix}\in \bfF^F|y\in \bigoplus_{h\in H}\Hom_k(\bfV'_{s(h)},\bfV''_{t(h)})\},$$
which is isomorphic to $\bigoplus_{h\in H}\Hom_{\bbF_q}(\bbF_{q^{\nu'_{s(h)}}},\bbF_{q^{\nu''_{t(h)}}})$. Consider the exact sequence of $\bbF_q$-vector spaces 
\begin{align*}
0\rightarrow &\Hom_{\cA}(M,N)\rightarrow \bigoplus_{i\in I}\Hom_{\bbF_q}(\bbF_{q^{\nu'_i}},\bbF_{q^{\nu''_i}})\\
\rightarrow &\bigoplus_{h\in H}\Hom_{\bbF_q}(\bbF_{q^{\nu'_{s(h)}}},\bbF_{q^{\nu''_{t(h)}}})\xrightarrow{\pi} \Ext_{\cA}^1(M,N)\rightarrow 0,
\end{align*}
see \cite[Proposition 3.14, Section 3.5]{Deng-Du-Parshall-Wang-2008}, we have $\pi^{-1}(\Ext^1_{\cA}(M,N)_L)\cong \kappa^{-1}(x',x'')\cap \cO_L$ such that
$$(\res^{\nu}_{\nu',\nu''}(1_{\cO_L}))(x',x'')={v_q}^{\sum_{i\in I}\nu'_i\nu''_i-\sum_{h\in H}\nu'_{s(h)}\nu''_{t(h)}}q^{\sum_{i\in I}\nu'_i\nu''_i}\frac{|\Ext^1_{\cA}(M,N)_L|}{|\Hom_{\cA}(M,N)|},$$
and so
$$\res^{\nu}_{\nu',\nu''}(1_{\cO_L})=\sum_{[M],[N]}{{v_q}}^{3\sum_{i\in I}\nu'_i\nu''_i-\sum_{h\in H}\nu'_{s(h)}\nu''_{t(h)}}\frac{|\Ext^1_{\cA}(M,N)_L|}{|\Hom_{\cA}(M,N)|}1_{\cO_M}\otimes 1_{\cO_N},$$
Therefore, $\Phi$ preserves the comultiplication.
\end{proof}

\section{Main theorem and application}\label{Main theorem and application}

For any $\alpha,\beta,\alpha',\beta'\in \bbN I^a$ satisfying $\gamma=\alpha+\beta=\alpha'+\beta'$, let $\cN$ be the set of $\lambda=(\alpha_1,\alpha_2,\beta_1,\beta_2)\in (\mathbb{N}I)^{4}$ satisfying $$\alpha=\alpha_1+\alpha_2,\beta=\beta_1+\beta_2,\alpha'=\alpha_1+\beta_1=\alpha',\beta'=\alpha_2+\beta_2,$$ 
the the cyclic group $\langle a\rangle$ acts on $\cN$ by $a.(\alpha_1,\alpha_2,\beta_1,\beta_2)=(a.\alpha_1,a.\alpha_2,a.\beta_1,a.\beta_2)$. We denote by $\cN^a\subset \cN$ the fix point set under $a$, that is, $\cN^a=\cN\cap (\bbN I^a)^{4}$.
For any $\lambda=(\alpha_1,\alpha_2,\beta_1,\beta_2)\in \cN$, we define $$\tau_\lambda:\bfE_{\alpha_1}\times \bfE_{\alpha_2}\times \bfE_{\beta_1}\times \bfE_{\beta_2}\rightarrow \bfE_{\alpha_1}\times \bfE_{\beta_1}\times \bfE_{\alpha_2}\times \bfE_{\beta_2}$$
to be the natural isomorphism switching the second and the third coordinates. It is clear that $a^*{\tau_\lambda}_!\cong {\tau_\lambda}_!a^*$, and we have the functor $\tilde{{\tau_\lambda}_!}$.

\begin{theorem}\label{main theorem}
For any $(A,\varphi)\in \tilde{\cD}_\alpha$ and $(B,\psi) \in \tilde{\cD}_{\beta}$, we have 
\begin{equation}\label{main}
\begin{aligned}
\Res_{\alpha' ,\beta'}^{\gamma}\Ind_{\alpha,\beta}^{\gamma}((A,\varphi)\boxtimes(B,\psi))\cong (C,\phi)\oplus \!\!\!\!\bigoplus_{\lambda=(\alpha_1,\alpha_2,\beta_1,\beta_2) \in \cN^a}\!\!\!\!\!\!(\Ind_{\alpha_1,\beta_1}^{\alpha'}\boxtimes \Ind_{\alpha_2,\beta_2}^{\beta'})\tilde{{\tau_\lambda}_!}\\
(\Res_{\alpha_1,\alpha_2}^{\alpha}(A,\varphi)\boxtimes \Res_{\beta_1,\beta_2}^{\beta}(B,\psi))[-(\alpha_2,\beta_1)](-\frac{(\alpha_2,\beta_1)}{2}),
\end{aligned}
\end{equation}
where $(C,\phi)\in \tilde{\cD}_{\alpha',\beta'}$ is traceless.
\end{theorem}

The proof will be given in the next subsection. Indeed, it is a refinement of the proof of \cite[Theorem 3.1]{fang2023parity} for the consideration about the automorphisms. We only need to prove the case that $A, B$ are simple perverse sheaves which must be pure, see \cite[Theorem 5.4.12]{Pramod-2021}. Suppose they are pure of weights $\omega, \omega'$ respectively.

\subsection{Proof of the theorem}

\subsubsection{The left hand side}\

We draw the following commutative diagram containing all data we will use in this subsection. The details will be introduced later.

\begin{diagram}[midshaft,size=2em]
\bfE_\alpha \times \bfE_\beta &\lTo^{p_1} &\bfE'^{\gamma}_{\alpha,\beta} &\rTo^{p_2} &\bfE''^{\gamma}_{\alpha,\beta} &\hEq &\bfE''^{\gamma}_{\alpha,\beta} &\rTo^{p_3} &\bfE_\gamma\\
& &\uTo^{\iota_\lambda''} &\square &\uTo^{\iota'\iota_\lambda'} & &\uTo^{\iota'} &\square &\uTo_{\iota}\\
& &\cQ_\lambda &\rTo^{{p_2}_\lambda'} &\tilde{\bfF}_\lambda &\rTo^{\iota_\lambda'} &\tilde{\bfF} &\rTo^{{p_3}'} &\bfF^{\gamma}_{\alpha',\beta'}\\
&&& &\dTo^{f_\lambda} & & & &\dTo_{\kappa}\\
&&& &(\bfE''\times \bfE'')_\lambda&&&\rTo^{{p_3}_\lambda} &\bfE_{\alpha'}\times \bfE_{\beta'}
\end{diagram}

By definition, the left hand side of the formula (\ref{main}) is equal to 
$$L=\tilde{\kappa_!}\tilde{\iota^*}\tilde{{p_3}_!}\tilde{{p_2}_\flat}\tilde{p_1^*}((A,\varphi)\boxtimes (B,\psi))[M](\frac{M}{2}),$$
where $M=\sum_{i\in I}\alpha_i\beta_i+\sum_{h\in H}\alpha_{s(h)}\beta_{t(h)}-\langle\alpha',\beta'\rangle.$

Let $\tilde{\bfF}=\bfE''^\gamma_{\alpha,\beta}\times_{\bfE_\gamma}\bfF^\gamma_{\alpha',\beta'}$ be the fiber product of $p_3:\bfE''^\gamma_{\alpha,\beta}\rightarrow \bfE_\gamma$ and $\iota:\bfF^\gamma_{\alpha',\beta'}\rightarrow \bfE_\gamma$ with the natural morphisms ${p_3}':\tilde{\bfF}\rightarrow \bfF^\gamma_{\alpha',\beta'}$ and $\iota':\tilde{\bfF}\rightarrow \bfE''^\gamma_{\alpha,\beta}$, then ${\iota^*}{{p_3}_!}\cong {{p_3}'_!}{\iota'^*}$ by base change. There is an automorphism $a:\tilde{\bfF}\rightarrow \tilde{\bfF}$ induced by the automorphisms $a$ on $\bfE''^\gamma_{\alpha,\beta}, \bfE_\gamma, \bfF^\gamma_{\alpha',\beta'}$ such that $\iota'a=a\iota', a^*\iota'^*=\iota'^*a^*$ and there is a cartesian diagram
\begin{diagram}[midshaft,size=2em]
\tilde{\bfF} &\rTo^{{p_3}'} &\tilde{\bfF}^\gamma_{\alpha',\beta'}\\
\dTo^a &\square &\dTo_a\\
\tilde{\bfF} &\rTo^{{p_3}'} &\tilde{\bfF}^\gamma_{\alpha',\beta'}.
\end{diagram}
By base change, we have $a^*{p_3}'_!\cong {p_3}'_!a^*$ and the functors  $\tilde{{p_3}'_!}, \tilde{\iota'^*}$ such that $\tilde{\iota^*}\tilde{{p_3}_!}\cong \tilde{{p_3}'_!}\tilde{\iota'^*}$, and so
\begin{equation}\label{first step}
L\cong \tilde{\kappa_!}\tilde{{p_3}'_!}\tilde{\iota'^*}\tilde{{p_2}_\flat}\tilde{p_1^*}((A,\varphi)\boxtimes (B,\psi))[M](\frac{M}{2}).
\end{equation}

Note that $\tilde{\bfF}$ consists of $(x,W)$, where $x\in \bfE_\gamma$ such that $\bfW^{\beta'}$ is $x$-stable, where $\bfW^{\beta'}\subset \bfV^\gamma$ is the fixed $I$-graded subspace appearing in the definition of $\Res^\gamma_{\alpha',\beta'}$, and $W\subset \bfV^\gamma$ is a $x$-stable, $I$-graded $k$-vector space of dimension vector $\beta$. For any $\lambda=(\alpha_1,\alpha_2,\beta_1,\beta_2)\in \cN$, let $\tilde{\bfF}_\lambda\subset \tilde{\bfF}$ be the locally closed subvariety consisting of $(x,W)$ such that the dimension vector of $W\cap\bfW^{\beta'}$ is $\beta_2$ and $\iota_\lambda':\tilde{\bfF}_\lambda\rightarrow \tilde{\bfF}$ be the inclusion, we denote by $(\bfE''\times \bfE'')_\lambda=\bfE''^{\alpha'}_{\alpha_1,\beta_1}\times \bfE''^{\beta'}_{\alpha_2,\beta_2}$ and ${p_3}_\lambda=p_3\times p_3:(\bfE''\times \bfE'')_\lambda\rightarrow \bfE_{\alpha'}\times \bfE_{\beta'}$, then there is a morphism $f_\lambda:\tilde{\bfF}_\lambda\rightarrow (\bfE''\times \bfE'')_\lambda$ defined by 
\begin{align*}
(x,W)\mapsto ((\rho_1.x_{\bfV^\gamma/\bfW^{\beta'}},\rho_1(W/W\cap\bfW^{\beta'})), (\rho_2.x_{\bfW^{\beta'}}, \rho_2(W\cap\bfW^{\beta'}))),
\end{align*}
where $\rho_1:\bfV^\gamma/\bfW^{\beta'}\rightarrow \bfV^{\alpha'},\rho_2:\bfW^{\beta'}\rightarrow \bfV^{\beta'}$ are the fixed $I$-graded linear isomorphisms appearing in the definition of $\Res^\gamma_{\alpha',\beta'}$, and  $W/W\cap \bfW^{\beta'}\cong W+\bfW^{\beta'}/\bfW^{\beta'}$ is viewed as a subspace of $\bfV^\gamma/\bfW^{\beta'}$, such that ${p_3}_\lambda f_\lambda=\kappa{p_3}'\iota_\lambda'$. By \cite[Lemma 3.2]{fang2023parity}, $f_\lambda$ is a locally trivial vector bundle of rank
$$r_{\lambda}=\sum_{h\in H}(\alpha_{1s(h)}\alpha_{2t(h)}+\alpha_{1s(h)}\beta_{2t(h)}+\beta_{1s(h)}\beta_{2t(h)})+\sum_{i\in I}\alpha_{2i}\beta_{1i}.$$
By \cite[Corollary 3.7]{fang2023parity}, we have 
\begin{equation}\label{base isomorphism}
\kappa_!{p_3}'_!\iota'^*{p_2}_\flat p_1^*(A\boxtimes B)[M](\frac{M}{2})\cong\bigoplus_{\lambda\in \cN}{{p_3}_\lambda}_!{f_\lambda}_!\iota_\lambda'^*\iota'^*{p_2}_\flat p_1^*(A\boxtimes B)[M](\frac{M}{2}).
\end{equation}

We divide $\cN=\bigsqcup_{s}\cN_s$ into the disjoint union of $\langle a\rangle$-orbits. For any $\cN_s$, let 
$$\tilde{\bfF}_s=\bigsqcup_{\lambda\in \cN_s}\tilde{\bfF}_\lambda,\  (\bfE''\times\bfE'')_s=\bigsqcup_{\lambda\in\cN_s}(\bfE''\times \bfE'')_\lambda$$
and $\iota_s':\tilde{\bfF}_s\rightarrow \tilde{\bfF}, f_s:\tilde{\bfF}_s\rightarrow (\bfE''\times\bfE'')_s, {p_3}_s:(\bfE''\times\bfE'')_s\rightarrow \bfE_{\alpha'}\times \bfE_{\beta'}$ be the morphisms assembled by $\iota_\lambda', f_\lambda, {p_3}_\lambda$ respectively, then ${p_3}_s f_s=\kappa{p_3}'\iota_s'$.

\begin{lemma}\label{weight argument}
For any $\langle a\rangle$-orbit $\cN_s$ in $\cN$, we have 
$${{p_3}_s}_!{f_s}_!\iota_s'^*\iota'^*{p_2}_\flat p_1^*(A\boxtimes B)[M](\frac{M}{2})\cong \bigoplus_{\lambda\in \cN_s}{{p_3}_\lambda}_!{f_\lambda}_!\iota_\lambda'^*\iota'^*{p_2}_\flat p_1^*(A\boxtimes B)[M](\frac{M}{2}).$$
\end{lemma}
\begin{proof}
For any $m\in \bbN$, let $\tilde{\bfF}_{s,m}\subset \tilde{\bfF}_s$ be the disjoint union of those $\tilde{\bfF}_\lambda\subset \tilde{\bfF}_s$ which are of dimension $m$, $\cN_{s,m}=\{\lambda\in \cN_s|\tilde{\bfF}_\lambda\subset \tilde{\bfF}_{s,m}\},(\bfE''\times\bfE'')_{s,m}=\bigsqcup_{\lambda\in \cN_{s,m}}(\bfE''\times \bfE'')_\lambda$ and $\iota'_{s,m}:\tilde{\bfF}_{s,m}\rightarrow \tilde{\bfF}, f_{s,m}:\tilde{\bfF}_{s,m}\rightarrow (\bfE''\times\bfE'')_{s,m}, {p_3}_{s,m}:(\bfE''\times\bfE'')_{s,m}\rightarrow \bfE_{\alpha'}\times \bfE_{\beta'}$ be the morphisms assembled by $\iota_\lambda', f_\lambda, {p_3}_\lambda$ respectively, then ${p_3}_{s,m} f_{s,m}=\kappa{p_3}'\iota'_{s,m}$. For convenience, we set $\tilde{\bfF}_{s,m}=\varnothing$ for $m<0$ and $X=\iota'^*{p_2}_\flat p_1^*(A\boxtimes B)$.\\
$\bullet$ For any $m\in \bbN, \lambda\in \cN_{s,m}$, the subvariety $\tilde{\bfF}_\lambda\subset \tilde{\bfF}_{s,m}$ is both open and closed, thus
\begin{equation}\label{open and closed}
{{p_3}_{s,m}}_!{f_{s,m}}_!\iota_{s,m}'^*(X)\cong \bigoplus_{\lambda\in \cN_{s,m}}{{p_3}_\lambda}_!{f_\lambda}_!\iota_\lambda'^*(X).
\end{equation}
$\bullet$ For any $m\in \bbN $, let $\tilde{\bfF}_{s,\leqslant m}=\bigsqcup_{m'\leqslant m}\tilde{\bfF}_{s,m'}$ and $\iota'_{s,\leqslant m}:\tilde{\bfF}_{s,\leqslant m}\rightarrow \tilde{\bfF}$ be the inclusion, then $\tilde{\bfF}_{s,\leqslant m-1}\subset \tilde{\bfF}_{s,\leqslant m}$ is closed with the open complement $\tilde{\bfF}_{s,m}$, then there is a canonical distinguished triangle
$$j_!j^*\iota_{s,\leqslant m}'^*(X)\rightarrow \iota_{s,\leqslant m}'^*(X)\rightarrow i_!i^*\iota_{s,\leqslant m}'^*(X),$$
where $j:\tilde{\bfF}_{s,m}\rightarrow \tilde{\bfF}_{s,\leqslant m}$ and $i:\tilde{\bfF}_{s,\leqslant m-1}\rightarrow \tilde{\bfF}_{s,\leqslant m}$ are the inclusions such that $\iota'_{s,\leqslant m}j=\iota'_{s,m}$ and $\iota'_{s,\leqslant m}i=\iota'_{s,\leqslant m-1}$. Applying the triangulated functor $\kappa_!{p_3}'_!{\iota'_{s,\leqslant m}}_!$, we obtain the distinguished triangle
\begin{align*}
\kappa_!{p_3}'_!{\iota'_{s,m}}_!\iota_{s,m}'^*(X)\rightarrow \kappa_!{p_3}'_!{\iota'_{s,\leqslant m}}_!\iota_{s,\leqslant m}'^*(X)\rightarrow \kappa_!{p_3}'_!{\iota'_{s,\leqslant m-1}}_!\iota_{s,\leqslant m-1}'^*(X)\xrightarrow{+1}.
\end{align*}
Applying the perverse cohomology functors ${{}^pH^t}$ for $t\in \bbZ$, we obtain the long exact sequences of perverse sheaves
\begin{equation}\label{long sequence}
\begin{aligned}
...&\rightarrow {{}^pH^{t-1}}(\kappa_!{p_3}'_!{\iota'_{s,\leqslant m-1}}_!\iota_{s,\leqslant m-1}'^*(X))\xrightarrow{\delta_t} {{}^pH^t}(\kappa_!{p_3}'_!{\iota'_{s,m}}_!\iota_{s,m}'^*(X))\\
&\rightarrow {{}^pH^t}(\kappa_!{p_3}'_!{\iota'_{s,\leqslant m}}_!\iota_{s,\leqslant m}'^*(X))\rightarrow {{}^pH^t}(\kappa_!{p_3}'_!{\iota'_{s,\leqslant m-1}}_!\iota_{s,\leqslant m-1}'^*(X))\rightarrow....
\end{aligned}
\end{equation}
We make an induction on $m$ to prove the following statements.\\
{\rm{(a)}} For any $m\in \bbN,t\in \bbZ$, the connecting morphism $\delta_t=0$.\\
{\rm{(b)}} For any $m\in \bbN$, the complex $\kappa_!{p_3}'_!{\iota'_{s,\leqslant m}}_!\iota_{s,\leqslant m}'^*(X)$ is pure of weight $\omega+\omega'$ and 
\begin{equation}\label{dimension inductive}
\kappa_!{p_3}'_!{\iota'_{s,\leqslant m}}_!\iota_{s,\leqslant m}'^*(X)\cong \kappa_!{p_3}'_!{\iota'_{s,m}}_!\iota_{s,m}'^*(X)\oplus \kappa_!{p_3}'_!{\iota'_{s,\leqslant m-1}}_!\iota_{s,\leqslant m-1}'^*(X).
\end{equation}
If $m<0$, the statements hold trivially. Assume the statements hold for $m-1$, then ${}^pH^{t-1}(\kappa_!{p_3}'_!{\iota'_{s,\leqslant n-1}}_!\iota_{s,\leqslant n-1}'^*(X))$ is pure of weight $\omega+\omega'+t-1$. By \cite[Corollary 3.5]{fang2023parity} and (\ref{open and closed}), the complex 
$$\kappa_!{p_3}'_!{\iota'_{s,m}}_!\iota_{s,m}'^*(X)={{p_3}_{s,m}}_!{f_{s,m}}_!\iota_{s,m}'^*(X)\cong \bigoplus_{\lambda\in \cN_{s,m}}{{p_3}_\lambda}_!{f_\lambda}_!\iota_\lambda'^*(X)$$
is pure of weight $\omega+\omega'$, and then ${{}^pH^t}(\kappa_!{p_3}'_!{\iota'_{s,m}}_!\iota_{s,m}'^*(X))$ is pure of weight $\omega+\omega'+t$. Hence $\delta_t$ is a morphism between two perverse sheaves of different weights which must be zero. This proves the statement {\rm{(a)}} for $m$. For any $t\in \bbZ$, by (\ref{long sequence}), we have the short exact sequence 
\begin{equation}\label{short sequence}
\begin{aligned}
0\rightarrow &{{}^pH^t}(\kappa_!{p_3}'_!{\iota'_{s,m}}_!\iota_{s,m}'^*(X))\rightarrow {{}^pH^t}(\kappa_!{p_3}'_!{\iota'_{s,\leqslant m}}_!\iota_{s,\leqslant m}'^*(X))\\
\rightarrow &{{}^pH^t}(\kappa_!{p_3}'_!{\iota'_{s,\leqslant m-1}}_!\iota_{s,\leqslant m-1}'^*(X))\rightarrow0,
\end{aligned}
\end{equation}
then ${{}^pH^t}(\kappa_!{p_3}'_!{\iota'_{s,\leqslant m}}_!\iota_{s,\leqslant m}'^*(X))$ is pure of weight $\omega+\omega'+t$, since ${{}^pH^t}(\kappa_!{p_3}'_!{\iota'_{s,m}}_!\iota_{s,m}'^*(X))$ and ${{}^pH^t}(\kappa_!{p_3}'_!{\iota'_{s,\leqslant m-1}}_!\iota_{s,\leqslant m-1}'^*(X))$ are pure of weight $\omega+\omega'+t$. Hence the complex $\kappa_!{p_3}'_!{\iota'_{s,\leqslant m}}_!\iota_{s,\leqslant m}'^*(X)$ is pure of weight $\omega+\omega'$. By \cite[Theorem 5.4.19]{Pramod-2021}, the short exact sequence (\ref{short sequence}) is about semisimple perverse sheaves which must split, and so 
\begin{align*}
{{}^pH^t}(\kappa_!{p_3}'_!{\iota'_{s,\leqslant m}}_!\iota_{s,\leqslant m}'^*(X))\cong {{}^pH^t}(\kappa_!{p_3}'_!{\iota'_{s,m}}_!\iota_{s,m}'^*(X))\oplus {{}^pH^t}(\kappa_!{p_3}'_!{\iota'_{s,\leqslant m-1}}_!\iota_{s,\leqslant m-1}'^*(X))
\end{align*}
for any $t\in \bbZ$. Therefore, 
$$\kappa_!{p_3}'_!{\iota'_{s,\leqslant m}}_!\iota_{s,\leqslant m}'^*(X)\cong \kappa_!{p_3}'_!{\iota'_{s,m}}_!\iota_{s,m}'^*(X)\oplus \kappa_!{p_3}'_!{\iota'_{s,\leqslant m-1}}_!\iota_{s,\leqslant m-1}'^*(X).$$
This proves the statement {\rm{(b)}} for $m$, as desired.\\
$\bullet$ If $m$ is large enough, then $\tilde{\bfF}_{s,\leqslant m}=\tilde{\bfF}_s, \iota'_{s,\leqslant m}=\iota_s'$ and $\cN_s=\bigsqcup_{0\leqslant m'\leqslant m}\cN_{s,m'}$. Moreover, by (\ref{open and closed}) and (\ref{dimension inductive}), we have 
\begin{align*}
&{{p_3}_s}_!{f_s}_!\iota_s'^*(X)=\kappa_!{p_3}'_!{\iota'_s}_!\iota_s'^*(X)=\kappa_!{p_3}'_!{\iota'_{s,\leqslant m}}_!\iota_{s,\leqslant m}'^*(X)\\
\cong &\bigoplus_{0\leqslant m'\leqslant m}\kappa_!{p_3}'_!{\iota'_{s,m'}}_!\iota_{s,m'}'^*(X)\cong \bigoplus_{0\leqslant m'\leqslant m}\bigoplus_{\lambda\in \cN_{s,m'}}{{p_3}_\lambda}_!{f_\lambda}_!{\iota_\lambda'}^*(X)\\
=&\bigoplus_{\lambda\in \cN_s}{{p_3}_\lambda}_!{f_\lambda}_!{\iota_\lambda'}^*(X),
\end{align*}
as desired.
\end{proof}

Note that the automorphism $a:\tilde{\bfF}\rightarrow \tilde{\bfF}$ restricts to an automorphism $a:\tilde{\bfF}_s\rightarrow \tilde{\bfF}_s$ such that $a(\tilde{\bfF}_\lambda)=\tilde{\bfF}_{a.\lambda}$, and there is an automorphism $a:(\bfE''\times\bfE'')_s\rightarrow (\bfE''\times\bfE'')_s$ defined by
$$a((x_1,W_1),(x_2,W_2))=((a(x_1),a_{\bfV^{\alpha'}}(W_1)),(a(x_2),a_{\bfV^{\beta'}}(W_2)))$$
for any $((x_1,W_1),(x_2,W_2))\in \bfE''^{\alpha'}_{\alpha_1,\beta_1}\times \bfE''^{\beta'}_{\alpha_2,\beta_2}$ such that $a((\bfE''\times \bfE'')_\lambda)=(\bfE''\times \bfE'')_{a.\lambda}$. By $\iota'_s a=a\iota_s'$ and $a^*\iota_s'^*=\iota_s'^*a^*$, we have the functor $\tilde{\iota_s'^*}$. By the cartesian diagrams
\begin{diagram}[midshaft,size=2em]
\tilde{\bfF}_s &\rTo^{f_s} &(\bfE''\times \bfE'')_s && &(\bfE''\times \bfE'')_s &\rTo^{{p_3}_s} &\bfE_{\alpha'}\times \bfE_{\beta'}\\
\dTo^a &\square &\dTo_a && &\dTo^a &\square &\dTo_a\\
\tilde{\bfF}_s &\rTo^{f_s} &(\bfE''\times \bfE'')_s && &(\bfE''\times \bfE'')_s &\rTo^{{p_3}_s} &\bfE_{\alpha'}\times \bfE_{\beta'}
\end{diagram}
and base change, we have $a^*{f_s}_!\cong {f_s}_!a^*, a^*{{p_3}_s}_!\cong {{p_3}_s}_!a^*$ and the functors $\tilde{{f_s}_!}, \tilde{{{p_3}_s}_!}$. We set 
$$(C_s,\phi_s)=\tilde{{{p_3}_s}_!}\tilde{{f_s}_!}\tilde{\iota_s'^*}\tilde{\iota'^*}\tilde{{p_2}_\flat}\tilde{p_1^*}((A,\varphi)\boxtimes (B,\psi))[M](\frac{M}{2}),$$
where $C_s={{p_3}_s}_!{f_s}_!\iota_s'^*\iota'^*{p_2}_\flat p_1^*(A\boxtimes B)[M](\frac{M}{2}), \phi_s:a^*(C_s)\rightarrow C_s$ is induced by the isomorphisms $\varphi:a^*(A)\rightarrow A,\psi:a^*(B)\rightarrow B$ together with  $a^*{{p_3}_s}_!\cong {{p_3}_s}_!a^*, a^*{f_s}_!\cong {f_s}_!a^*, a^*\iota_s'^*=\iota_s'^*a^*, a^*\iota'^*=\iota'^*a^* ,a^*{p_2}_\flat\cong {p_2}_\flat a^*, a^*p_1^*=p_1^*a^*$.

\begin{lemma}\label{traceless term}
For any $\langle a\rangle$-orbit $\cN_s$ in $\cN$, if $|\cN_s|\geqslant 2$, then $(C_s,\phi_s)$ is traceless.
\end{lemma}
\begin{proof}
Suppose $|\cN_s|=t$ and $\cN_s=\{\lambda, a.\lambda,...a^{t-1}.\lambda\}$, so $t|n$. By Lemma \ref{weight argument}, we have 
$${{p_3}_s}_!{f_s}_!\iota_s'^*\iota'^*{p_2}_\flat p_1^*(A\boxtimes B)[M](\frac{M}{2})\cong \bigoplus_{0\leqslant t'\leqslant t-1}\! {{p_3}_{(a^{t'}.\lambda)}}_!{f_{(a^{t'}.\lambda)}}_!\iota_{(a^{t'}.\lambda)}'^*\iota'^*{p_2}_\flat p_1^*(A\boxtimes B)[M](\frac{M}{2}).$$
For convenience, we denote by $Y_{t'}={{p_3}_{(a^{t'}.\lambda)}}_!{f_{(a^{t'}.\lambda)}}_!\iota_{(a^{t'}.\lambda)}'^*\iota'^*{p_2}_\flat p_1^*(A\boxtimes B)[M](\frac{M}{2})$ for $t'\in \bbZ/t\bbZ$. Note that $\iota_{(a^{t'+1}.\lambda)}'a=a\iota_{(a^{t'}.\lambda)}'$ and there are cartesian diagrams
\begin{diagram}[midshaft,size=2em]
\tilde{\bfF}_{(a^{t'}.\lambda)} &\rTo^{f_{(a^{t'}.\lambda)}} &(\bfE''\times \bfE'')_{(a^{t'}.\lambda)} && &(\bfE''\times \bfE'')_{(a^{t'}.\lambda)} &\rTo^{{p_3}_{(a^{t'}.\lambda)}} &\bfE_{\alpha'}\times \bfE_{\beta'}\\
\dTo^a &\square &\dTo_a && &\dTo^a &\square &\dTo_a\\
\tilde{\bfF}_{(a^{t'+1}.\lambda)} &\rTo^{f_{(a^{t'+1}.\lambda)}} &(\bfE''\times \bfE'')_{(a^{t'+1}.\lambda)}, && &(\bfE''\times \bfE'')_{(a^{t'+1}.\lambda)} &\rTo^{{p_3}_{(a^{t'+1}.\lambda)}} &\bfE_{\alpha'}\times \bfE_{\beta'},
\end{diagram}
we have $a^*\iota_{(a^{t'+1}.\lambda)}'^*=\iota_{(a^{t'}.\lambda)}'^*a^*, a^*{f_{(a^{t'+1}.\lambda)}}_!\cong{f_{(a^{t'}.\lambda)}}_!a^*,a^*{{p_3}_{(a^{t'+1}.\lambda)}}_!\cong{{p_3}_{(a^{t'}.\lambda)}}_!a^*$ by base change, such that 
\begin{align*}
a^*Y_{t'+1}=&a^*{{p_3}_{(a^{t'+1}.\lambda)}}_!{f_{(a^{t'+1}.\lambda)}}_!\iota_{(a^{t'+1}.\lambda)}'^*\iota'^*{p_2}_\flat p_1^*(A\boxtimes B)[M](\frac{M}{2})\\
\cong&{{p_3}_{(a^{t'}.\lambda)}}_!{f_{(a^{t'}.\lambda)}}_!\iota_{(a^{t'}.\lambda)}'^*\iota'^*{p_2}_\flat p_1^*(a^*(A)\boxtimes a^*(B))[M](\frac{M}{2})\\
\cong &{{p_3}_{(a^{t'}.\lambda)}}_!{f_{(a^{t'}.\lambda)}}_!\iota_{(a^{t'}.\lambda)}'^*\iota'^*{p_2}_\flat p_1^*(A\boxtimes B)[M](\frac{M}{2})=Y_{t'},
\end{align*}
where the second isomorphism is induced by $\varphi:a^*(A)\rightarrow A$ and $ \psi:a^*(B)\rightarrow B$. Therefore, under the identification $C_s\cong \bigoplus_{0\leqslant t'\leqslant t}Y_{t'}$, the isomorphism $\phi_s$ corresponds to the isomorphism taking $Y_{t'+1}$ onto $Y_t'$ for $0\leqslant t'\leqslant t-2$ and taking $Y_0$ onto $Y_{t-1}$, and so $(C_s,\phi_s)$ is traceless.
\end{proof}

We set $$(C,\phi)=\bigoplus_{s:|\cN_s|\geqslant 2} (C_s,\phi_s),$$ then it is traceless. For any $\langle a\rangle$-orbit $\cN_s$ in $\cN$ satisfying $|\cN_s|=1$, suppose $\cN_s=\{\lambda\}$, then $\lambda\in \cN^a$ and  $\tilde{\bfF}_\lambda=\tilde{\bfF}_s,(\bfE''\times \bfE'')_\lambda=(\bfE''\times \bfE'')_s,\iota_\lambda'=\iota_s',f_\lambda=f_s,{p_3}_\lambda={p_3}_s$. By (\ref{first step}), (\ref{base isomorphism}), Lemma \ref{traceless term} and Lemma \ref{split criterion}, we obtain the following corollary.

\begin{corollary}\label{traceless term corollary}
The left hand side of the formula {\rm{(\ref{main})}} is isomorphic to
$$(C,\phi)\oplus \bigoplus_{\lambda\in \cN^a}\tilde{{{p_3}_\lambda}_!}\tilde{{f_\lambda}_!}\tilde{\iota_\lambda'^*}\tilde{\iota'^*}\tilde{{p_2}_\flat}\tilde{p_1^*}((A,\varphi)\boxtimes (B,\psi))[M](\frac{M}{2}),$$
where $(C,\phi)\in \tilde{\cD}_{\alpha',\beta'}$ is traceless.
\end{corollary}

For any $\lambda\in \cN^a$, let $\cQ_\lambda=\bfE'^{\gamma}_{\alpha,\beta}\times_{\bfE''^{\gamma}_{\alpha,\beta}}\tilde{\bfF}_\lambda$ be the fiber product of $p_2:\bfE'^{\gamma}_{\alpha,\beta}\rightarrow \bfE''^{\gamma}_{\alpha,\beta}$ and $\iota'\iota_\lambda':\tilde{\bfF}_{\lambda}\rightarrow \bfE''^{\gamma}_{\alpha,\beta}$ with the natural morphism ${\iota_\lambda}'':\cQ_\lambda\rightarrow \bfE'^{\gamma}_{\alpha,\beta}$ and ${p_2}_\lambda':\cQ_\lambda\rightarrow \tilde{\bfF}_\lambda$, then ${p_2}_\lambda'^*\iota_\lambda'^*\iota'^*={\iota''_\lambda}^*p_2^*$. Since ${p_2}_\lambda'$ is a $\bfG_{\alpha}\times \bfG_{\beta}$-principal bundle, ${p_2}_\lambda'^*$ is an equivalence of categories with the quasi-inverse denoted by ${{p_2}_\lambda'}_\flat$, then $\iota_\lambda'^*\iota'^*{p_2}_\flat\cong {{p_2}_\lambda'}_\flat{\iota''_\lambda}^*$. There is an automorphism $a:\cQ_\lambda\rightarrow \cQ_\lambda$ induced by the automorphisms $a$ on $\bfE'^{\gamma}_{\alpha,\beta}, \bfE''^{\gamma}_{\alpha,\beta}, \tilde{\bfF}_\lambda$ such that ${p_2}_\lambda' a=a{p_2}_\lambda', \iota''_\lambda a=a\iota''_\lambda$ and $a^*{p_2}_\lambda'^*={p_2}_\lambda'^*a^*,a^*{{p_2}_\lambda'}_\flat\cong {{p_2}_\lambda'}_\flat a^*,a^*{\iota_\lambda''}^*={\iota_\lambda''}^*a^*$, and there are functors $\tilde{{{p_2}_\lambda'}_\flat},\tilde{{\iota_\lambda''}^*}$ such that $\tilde{\iota_\lambda'^*}\tilde{\iota'^*}\tilde{{p_2}_\flat}\cong \tilde{{{p_2}_\lambda'}_\flat}\tilde{{\iota_\lambda''}^*}$. By Corollary \ref{traceless term corollary}, we obtain the following proposition to end this subsection.

\begin{proposition}\label{left hand side result}
The left hand side of the formula {\rm{(\ref{main})}} is isomorphic to
$$(C,\phi)\oplus \bigoplus_{\lambda\in \cN^a}\tilde{{{p_3}_\lambda}_!}\tilde{{f_\lambda}_!}\tilde{{{p_2}_\lambda'}_\flat}\tilde{{\iota_\lambda''}^*}\tilde{p_1^*}((A,\varphi)\boxtimes (B,\psi))[M](\frac{M}{2}),$$
where $(C,\phi)\in \tilde{\cD}_{\alpha',\beta'}$ is traceless.
\end{proposition}

\subsubsection{The right hand side}\

We draw the following commutative diagram containing all data we will use in this subsection. The details will be introduced later.

\begin{diagram}[midshaft,size=2em]
\bfE_\alpha\times \bfE_\beta\\
\uTo^{\iota_\lambda}\\
(\bfF\times \bfF)_\lambda &\lTo{{p_1}_\lambda'} &\cO_\lambda\\
\dTo^{\kappa_\lambda}\\
\bfE_{\alpha_1}\times \bfE_{\alpha_2}\times \bfE_{\beta_1}\times \bfE_{\beta_2} &\square &\dTo_{{{\kappa}_\lambda}'}\\
\dTo^{\tau_\lambda}\\
\bfE_{\alpha_1}\times \bfE_{\beta_1}\times \bfE_{\alpha_2}\times \bfE_{\beta_2} &\lTo^{{p_1}_\lambda} &(\bfE'\times \bfE')_\lambda &\rTo^{{p_2}_\lambda} &(\bfE''\times \bfE'')_\lambda &\rTo^{{p_3}_\lambda} &\bfE_{\alpha'}\times \bfE_{\beta'}
\end{diagram}

For any $\lambda=(\alpha_1,\alpha_2,\beta_1,\beta_2)\in \cN^a$, we denote by 
\begin{align*}
&(\bfF\times \bfF)_\lambda=\bfF^\alpha_{\alpha_1,\alpha_2}\times \bfF^\beta_{\beta_1,\beta_2},\ (\bfE''\times \bfE'')_\lambda=\bfE'^{\alpha'}_{\alpha_1,\beta_1}\times \bfE'^{\beta'}_{\alpha_2,\beta_2}\\
&\iota_\lambda=(\iota\times \iota):(\bfF\times \bfF)_\lambda\rightarrow \bfE_\alpha\times \bfE_\beta,\\&\kappa_\lambda=(\kappa\times \kappa):(\bfF\times \bfF)_\lambda\rightarrow\bfE_{\alpha_1}\times \bfE_{\alpha_2}\times \bfE_{\beta_1}\times \bfE_{\beta_2},\\
&{p_1}_\lambda=(p_1\times p_1):(\bfE'\times \bfE')_\lambda\rightarrow \bfE_{\alpha_1}\times \bfE_{\beta_1}\times \bfE_{\alpha_2}\times \bfE_{\beta_2}, \\
&{p_2}_\lambda=(p_2\times p_2):(\bfE'\times \bfE')_\lambda\rightarrow(\bfE''\times \bfE'')_\lambda.
\end{align*}

By definition, the right hand side of the formula (\ref{main}) is equal to 
$$(C,\phi)\oplus\bigoplus_{\lambda\in \cN^a}\tilde{{{p_3}_\lambda}_!}\tilde{{{p_2}_\lambda}_\flat}\tilde{{p_1}_\lambda^*}\tilde{{\tau_\lambda}_!}\tilde{{\kappa_\lambda}_!}\tilde{\iota_\lambda^*}((A,\varphi)\boxtimes (B,\psi))[N_\lambda-(\alpha_2,\beta_1)](\frac{N_\lambda-(\alpha_2,\beta_1)}{2}),$$
where $$N_\lambda=-\langle \alpha_1,\alpha_2\rangle-\langle \beta_1,\beta_2\rangle+\sum_{i\in I}(\alpha_{1i}\beta_{1i}+\alpha_{2i}\beta_{2i})+\sum_{h\in H}(\alpha_{1s(h)}\beta_{1t(h)}+\alpha_{2s(h)}\beta_{2t(h)}).$$

Let $\cO_\lambda=(\bfF\times \bfF)_\lambda\times_{\bfE_{\alpha_1}\times \bfE_{\beta_1}\times \bfE_{\alpha_2}\times \bfE_{\beta_2}}(\bfE'\times \bfE')_\lambda$ be the fiber product of $\tau_\lambda\kappa_\lambda:(\bfF\times \bfF)_\lambda\rightarrow \bfE_{\alpha_1}\times \bfE_{\beta_1}\times \bfE_{\alpha_2}\times \bfE_{\beta_2}$ and ${p_1}_\lambda:(\bfE'\times \bfE')_\lambda\rightarrow \bfE_{\alpha_1}\times \bfE_{\beta_1}\times \bfE_{\alpha_2}\times \bfE_{\beta_2}$ with the natural morphism ${p_1}_\lambda':\cO_\lambda\rightarrow (\bfF\times \bfF)_\lambda$ and ${\kappa_\lambda}':\cO_\lambda\rightarrow (\bfE'\times \bfE')_\lambda$, then ${p_1}_\lambda^*{\tau_\lambda}_!{\kappa_\lambda}_!\cong {{\kappa}_\lambda}'_!{p_1}_\lambda'^*$. There is an automorphism $a:\cO_\lambda\rightarrow \cO_\lambda$ induced by the automorphisms $a$ on $(\bfF\times \bfF)_\lambda, \bfE_{\alpha_1}\times \bfE_{\beta_1}\times \bfE_{\alpha_2}\times \bfE_{\beta_2},(\bfE'\times \bfE')_\lambda$ such that ${p_1}_\lambda' a=a{p_1}_\lambda', a^*{p_1}_\lambda'^*={p_1}_\lambda'^*a^*$ and there is a cartesian diagram 
\begin{diagram}[midshaft,size=2em]
\cO_\lambda &\rTo^{{\kappa_\lambda}'} &(\bfE'\times \bfE')_\lambda\\
\dTo^a &\square &\dTo_a\\
\cO_\lambda &\rTo^{{\kappa_\lambda}'} &(\bfE'\times \bfE')_\lambda.
\end{diagram}
By base change, we have $a^*{\kappa_\lambda}'_!\cong {\kappa_\lambda}'_!a^*$ and the functors $\tilde{{\kappa_\lambda}'_!},\tilde{{p_1}_\lambda'^*}$ such that $\tilde{{p_1}_\lambda^*}\tilde{{\tau_\lambda}_!}\tilde{{\kappa_\lambda}_!}\cong \tilde{{{\kappa}_\lambda}'_!}\tilde{{p_1}_\lambda'^*}$, and so we obtain the following proposition to end this subsection.

\begin{proposition}\label{right hand side result}
The right hand side of the formula {\rm{(\ref{main})}} is isomorphic to
$$(C,\phi)\oplus\bigoplus_{\lambda\in \cN^a}\tilde{{{p_3}_\lambda}_!}\tilde{{{p_2}_\lambda}_\flat}\tilde{{{\kappa}_\lambda}'_!}\tilde{{p_1}_\lambda'^*}\tilde{\iota_\lambda^*}((A,\varphi)\boxtimes (B,\psi))[N_\lambda-(\alpha_2,\beta_1)](\frac{N_\lambda-(\alpha_2,\beta_1)}{2})$$
\end{proposition}

\subsubsection{Connection between two sides}\

We draw the following commutative diagram containing all data we will use in this subsection. The details will be introduced later.

\begin{diagram}[midshaft,size=2em]
\bfE_{\alpha}\times \bfE_{\beta} && &\lTo^{p_1} &\bfE'^{\gamma}_{\alpha,\beta}\\
&&&&\uTo^{\iota_\lambda''}\\
\uTo^{\iota_\lambda}&&&&\cQ_\lambda &\rTo^{{p_2}_\lambda'} &\tilde{\bfF}_\lambda\\
&&&&\uTo^{\xi_\lambda} & &\vEq\\
(\bfF\times\bfF)_\lambda &\lTo^{{p_1}_\lambda'} &\cO_\lambda &\lTo^{\zeta_\lambda} &\cP_\lambda &\rTo^{{p_2}_\lambda''} &\tilde{\bfF}_\lambda\\
&&\dTo^{\kappa_\lambda'} & &\dTo^{f_\lambda'} &\square &\dTo_{f_\lambda}\\
&&(\bfE'\times \bfE')_\lambda &\hEq &(\bfE'\times \bfE')_\lambda &\rTo^{{p_2}_\lambda} &(\bfE''\times \bfE'')_\lambda
\end{diagram}

By Proposition \ref{left hand side result} and \ref{right hand side result}, we need to prove
\begin{align*}
&\tilde{{f_\lambda}_!}\tilde{{{p_2}_\lambda'}_\flat}\tilde{{\iota_\lambda''}^*}\tilde{p_1^*}((A,\varphi)\boxtimes (B,\psi))[M](\frac{M}{2})\\
\cong &\tilde{{{p_2}_\lambda}_\flat}\tilde{{{\kappa}_\lambda}'_!}\tilde{{p_1}_\lambda'^*}\tilde{\iota_\lambda^*}((A,\varphi)\boxtimes (B,\psi))[N_\lambda-(\alpha_2,\beta_1)](\frac{N_\lambda-(\alpha_2,\beta_1)}{2})
\end{align*}
for any $\lambda\in \cN^a$.

Let $\cP_\lambda=\tilde{\bfF}_\lambda\times_{(\bfE''\times\bfE'')_\lambda}(\bfE'\times\bfE')_\lambda$ be the fiber product of $f_\lambda:\tilde{\bfF}_\lambda\rightarrow (\bfE''\times\bfE'')_\lambda$ and ${p_2}_\lambda:(\bfE'\times\bfE')_\lambda\rightarrow (\bfE''\times\bfE'')_\lambda$ with the natural morphisms $f_\lambda':\cP_\lambda\rightarrow (\bfE'\times\bfE')_\lambda$ and ${p_2}_\lambda'':\cP_\lambda\rightarrow \tilde{\bfF}_\lambda$, then ${p_2}_\lambda^*{f_\lambda}_!\cong {f_\lambda'}_!{p_2}_\lambda''^*$ by base change. Since ${p_2}_\lambda''$ is a $\bfG_{\alpha_1}\times \bfG_{\beta_1}\times \bfG_{\alpha_2}\times\bfG_{\beta_2}$-bundle, ${p_2}_\lambda''^*$ is an equivalence of categories with the quasi-inverse denoted by ${{p_2}_\lambda''}_\flat$, then ${f_\lambda}_!{{p_2}_\lambda''}_\flat\cong {{p_2}_\lambda}_\flat {f_{\lambda'}}_!$. There is an automorphism $a:\cP_\lambda\rightarrow \cP_\lambda$ induced by the automorphisms $a$ on $\tilde{\bfF}_\lambda, (\bfE''\times\bfE'')_\lambda, (\bfE'\times\bfE')_\lambda$ such that ${p_2}_\lambda''a=a{p_2}_\lambda'', a^*{p_2}_\lambda''^*={p_2}_\lambda''^*a^*,a^*{{p_2}_\lambda''}_\flat\cong {p_2}_\lambda''a^*$ and there is a cartesian diagram
\begin{diagram}[midshaft,size=2em]
\cP_\lambda &\rTo^{f_\lambda'} &(\bfE'\times\bfE')_\lambda\\
\dTo^a &\square &\dTo_a\\
\cP_\lambda &\rTo^{f_\lambda'} &(\bfE'\times\bfE')_\lambda.
\end{diagram}
By base change, we have $a^*{f_\lambda'}_!\cong {f_\lambda'}_!a^*$ and the functors $\tilde{{f_\lambda'}_!},\tilde{{{p_2}_\lambda''}_\flat}$ such that $\tilde{{f_\lambda}_!}\tilde{{{p_2}_\lambda''}_\flat}\cong \tilde{{{p_2}_\lambda}_\flat}\tilde{{f_\lambda'}_!}$.

By \cite[Lemma 3.12]{fang2023parity}, there is a smooth morphism $\zeta_\lambda:\cP_\lambda\rightarrow \cO_\lambda$ with connected fibers of dimension 
$$r'_\lambda=r_\lambda-\sum_{h\in H}(\alpha_{1s(h)}\alpha_{2t(h)}+\beta_{1s(h)}\beta_{2t(h)})$$
such that $f_\lambda'=\kappa_\lambda'\zeta_\lambda$. Then $\zeta_\lambda a=a\zeta_\lambda, a^*\zeta_\lambda^*=\zeta_\lambda^*a^*$ and there is a cartesian diagram
\begin{diagram}[midshaft,size=2em]
\cP_\lambda &\rTo^{\zeta_\lambda} &\cO_\lambda\\
\dTo^a &\square &\dTo_a\\
\cP_\lambda &\rTo^{\zeta_\lambda} &\cO_\lambda.
\end{diagram}
By base change, we have $a^*{\zeta_\lambda}_!\cong{\zeta_\lambda}_!a^*$ and the functors $\tilde{{\zeta_\lambda}_!}, \tilde{\zeta_\lambda^*}$ such that $\tilde{{f_\lambda'}_!}=\tilde{{\kappa_\lambda'}_!}\tilde{{\zeta_\lambda}_!}$. Since $\zeta_\lambda$ is smooth with connected fibers, by \cite[Proposition 3.6.1, Theorem 3.6.6]{Pramod-2021}, $\zeta_\lambda^!=\zeta_\lambda^*[2r'_\lambda](r'_\lambda)$ is fully faithful such that
$${\zeta_\lambda}_!\zeta_\lambda^*\cong{\zeta_\lambda}_!\zeta_\lambda^![-2r'_\lambda](-r'_\lambda)\cong[-2r'_\lambda](-r'_\lambda)$$
and so $\tilde{{\zeta_\lambda}_!}\tilde{\zeta_\lambda^*}\cong [-2r'_\lambda](-r'_\lambda)$.

By Lemma \cite[Lemma 3.13]{fang2023parity}, there is a morphism $\xi_\lambda:\cP_\lambda\rightarrow \cQ_\lambda$ such that $p_1\iota_\lambda''\xi_\lambda=\iota_\lambda {p_1}_\lambda' \zeta_\lambda$ and ${p_2}_\lambda'\xi_\lambda={p_2}_\lambda''$. Then $\xi_\lambda a=a\xi_\lambda, a^*\xi_\lambda^*=\xi_\lambda^*a^*$, and we have the functor $\tilde{\xi_\lambda^*}$ such that $\tilde{\xi_\lambda^*}\tilde{\iota_\lambda''^*}\tilde{p_1^*}=\tilde{\zeta_\lambda^*}\tilde{{p_1}_\lambda'^*}\tilde{\iota_\lambda^*}$ and $\tilde{\xi_\lambda^*}\tilde{{p_2}_\lambda'^*}=\tilde{{p_2}_\lambda''^*}, \tilde{{{p_2}_\lambda'}_\flat}\cong\tilde{{{p_2}_\lambda''}_\flat}\tilde{\xi_\lambda^*}.$

\begin{proof}[Proof of the Theorem \ref{main theorem}]
By Proposition \ref{left hand side result} and \ref{right hand side result}, we need to prove that
\begin{align*}
&\tilde{{f_\lambda}_!}\tilde{{{p_2}_\lambda'}_\flat}\tilde{{\iota_\lambda''}^*}\tilde{p_1^*}((A,\varphi)\boxtimes (B,\psi))[M](\frac{M}{2})\\
\cong &\tilde{{{p_2}_\lambda}_\flat}\tilde{{{\kappa}_\lambda}'_!}\tilde{{p_1}_\lambda'^*}\tilde{\iota_\lambda^*}((A,\varphi)\boxtimes (B,\psi))[N_\lambda-(\alpha_2,\beta_1)](\frac{N_\lambda-(\alpha_2,\beta_1)}{2})
\end{align*}
for any $\lambda\in \cN^a$. Indeed, we have
\begin{align*}
&\tilde{{f_\lambda}_!}\tilde{{{p_2}_\lambda'}_\flat}\tilde{{\iota_\lambda''}^*}\tilde{p_1^*}((A,\varphi)\boxtimes (B,\psi))[M](\frac{M}{2})\\
\cong&\tilde{{f_\lambda}_!}\tilde{{{p_2}_\lambda''}_\flat}\tilde{\xi_\lambda^*}\tilde{{\iota_\lambda''}^*}\tilde{p_1^*}((A,\varphi)\boxtimes (B,\psi))[M](\frac{M}{2})\\
\cong&\tilde{{{p_2}_\lambda}_\flat}\tilde{{f_\lambda'}_!}\tilde{\xi_\lambda^*}\tilde{{\iota_\lambda''}^*}\tilde{p_1^*}((A,\varphi)\boxtimes (B,\psi))[M](\frac{M}{2})\\
\cong&\tilde{{{p_2}_\lambda}_\flat}\tilde{{\kappa_\lambda'}_!}\tilde{{\zeta_\lambda}_!}\tilde{\xi_\lambda^*}\tilde{{\iota_\lambda''}^*}\tilde{p_1^*}((A,\varphi)\boxtimes (B,\psi))[M](\frac{M}{2})\\
=&\tilde{{{p_2}_\lambda}_\flat}\tilde{{\kappa_\lambda'}_!}\tilde{{\zeta_\lambda}_!}\tilde{\zeta_\lambda^*}\tilde{{p_1}_\lambda'^*}\tilde{\iota_\lambda^*}((A,\varphi)\boxtimes (B,\psi))[M](\frac{M}{2})\\
\cong &\tilde{{{p_2}_\lambda}_\flat}\tilde{{{\kappa}_\lambda}'_!}\tilde{{p_1}_\lambda'^*}\tilde{\iota_\lambda^*}((A,\varphi)\boxtimes (B,\psi))[M-2r'_\lambda](\frac{M-2r'_\lambda}{2}).
\end{align*}
It is routine to check that $M-2r'_\lambda=N_\lambda-(\alpha_2,\beta_1)$, see \cite[Subsection 3.4]{{fang2023parity}}. 
\end{proof}
   
\subsection{Application}\

In this subsection, we take the fixed square root $v_q$ of $q$ appearing in section \ref{Hall algebra} and subsection \ref{Hall algebra via function} to be $-\sqrt{q}$.

By Theorem \ref{sheaf-function correspondence}, for any $\nu,\nu',\nu''\in \bbN I^a$ satisfying $\nu=\nu'+\nu''$, we have 
\begin{align*}
&\chi_{\Ind^\nu_{\nu',\nu''}((A,\varphi)\boxtimes(B,\psi))}=\ind^{\nu}_{\nu',\nu''}(\chi_{(A,\varphi)}\otimes \chi_{(B,\psi)}),\\
&\chi_{\Res^\nu_{\nu',\nu''}(C,\phi)}=\res^\nu_{\nu',\nu''}(\chi_{(C,\phi)})
\end{align*}
for any $(A,\varphi)\in \tilde{\cD}_{\nu'},(B,\psi)\in \tilde{\cD}_{\nu''},(C,\phi)\in \tilde{\cD}_\nu$.

By Theorem \ref{main theorem} and Lemma \ref{chitraceless}, for any $\alpha,\beta\in \bbN I^a$, we have 
\begin{align*}
\sum_{\alpha',\beta':\alpha'+\beta'=\alpha+\beta}\res^{\alpha+\beta}_{\alpha',\beta'}\ind^{\alpha+\beta}_{\alpha,\beta}(\chi_{(A,\varphi)}\otimes\chi_{(B,\psi)})
=\sum_{\alpha',\beta':\alpha'+\beta'=\alpha+\beta}\sum_{\lambda=(\alpha_1,\alpha_2,\beta_1,\beta_2)\in \cN^a}\\v_q^{(\alpha_2,\beta_1)}(\ind^{\alpha'}_{\alpha_1,\beta_1}\otimes \ind^{\beta'}_{\alpha_2,\beta_2}){\tau_{\lambda}}_!(\res^\alpha_{\alpha_1,\alpha_2}(\chi_{(A,\varphi)})\otimes \res^\beta_{\beta_1,\beta_2}(\chi_{(B,\psi)}))
\end{align*}
for any $(A,\varphi)\in \tilde{\cD}_{\alpha},(B,\psi)\in \tilde{\cD}_{\beta}$, that is, 
\begin{equation}\label{tilde(Delta)alghom}
\tilde{\Delta}(\chi_{(A,\varphi)}\,\,\tilde{*}\,\,\chi_{(B,\psi)})=\tilde{\Delta}(\chi_{(A,\varphi)})\,\,\tilde{*}\,\,\tilde{\Delta}(\chi_{(B,\psi)}),
\end{equation}
where $\tilde{*}$ is the multiplication and $\tilde{\Delta}$ is the comultiplication on $\tilde{\cH}(\cA)$ for $\cA=\rep_k^{F_{Q,a}}(Q)$, see subsection \ref{Hall algebra via function}, and the multiplication on $\tilde{\cH}(\cA)\otimes \tilde{\cH}(\cA)$ is defined by 
$$(1_{\cO_{M_1}}\otimes 1_{\cO_{M_2}})*(1_{\cO_{N_1}}\otimes 1_{\cO_{N_2}})={v_q}^{(\hat{M_2},\hat{N_1})}(1_{\cO_{M_1}}*1_{\cO_{N_1}})\otimes (1_{\cO_{M_2}}*1_{\cO_{N_2}}).$$

\begin{lemma}
For any $\nu\in \bbN I^a$ and $\bfG^F_\nu$-orbit $\cO=\bfG^F_\nu.x$ in $\bfE^F_\nu$, there exists finitely many $(A_s,\varphi_s)\in \tilde{\cD}_{\nu}$ for $s=1,...,m$ such that the characteristic function $1_\cO$ can be written as a $\bbC$-linear combination of $\chi_{(A_s,\varphi_s)}$ for $s=1,...,m$.
\end{lemma} 
\begin{proof}
We make an induction on the set of $\bfG^F_\nu$-orbits in $\bfE^F_\nu$ with respect to the partial order defined by the degeneration, that is, $\cO_1\preccurlyeq\cO_2$ if and only if their corresponding $\bfG_\nu$-orbits $\hat{\cO_1},\hat{\cO_2}$ in $\bfE_\nu$ satisfy $\hat{\cO_1}\subset \overline{\hat{\cO_2}}$. Note that the minimal orbits are closed and corresponding to the isomorphism classes of semisimple representations. 

Let $\hat{\cO}=\bfG_\nu.x$ be the $\bfG_\nu$-orbit of $x$ in $\bfE_\nu$ and $j:\hat{\cO}\rightarrow \bfE_\nu$ be the inclusion, consider the complexes 
$$S_{\hat{\cO}}=j_!(\overline{\bbQ}_l|_{\hat{\cO}}),\ I_{\hat{\cO}}=j_{!*}(\overline{\bbQ}_l|_{\hat{\cO}}),$$
where $I_{\hat{\cO}}$ is a $\bfG_\nu$-equivariant perverse sheaf whose support is $\overline{\hat{\cO}}$ and whose restriction to $\hat{\cO}$ is $\overline{\bbQ}_l|_{\hat{\cO}}$. Note that $\hat{\cO}$ is stable under $\tilde{F}:\bfE_\nu\rightarrow \bfE_\nu, a:\bfE_\nu\rightarrow \bfE_\nu$, there are natural isomorphisms $\xi:\tilde{F}^*(S_{\hat{\cO}})\rightarrow S_{\hat{\cO}},\varphi:a^*(S_{\hat{\cO}})\rightarrow S_{\hat{\cO}}$ such that $\chi_{(S_{\hat{\cO}},\varphi)}=1_{\cO}$. \\
$\bullet$ If $\cO$ is minimal, then $(S_{\hat{\cO}},\varphi)\in \tilde{\cD}_{\nu}$ is as desired. \\
$\bullet$ Otherwise, $\xi, \varphi$ induce isomorphisms $\overline{\xi}:\tilde{F}^*(I_{\hat{\cO}})\rightarrow I_{\hat{\cO}}, \overline{\varphi}:a^*(I_{\hat{\cO}})\rightarrow I_{\hat{\cO}}$ such that $(I_{\hat{\cO}},\overline{\varphi})\in \tilde{\cD}_{\nu}$ and $\chi_{(I_{\hat{\cO}},\overline{\varphi})}\in 1_{\cO}+\sum_{t=1}^l\bbC 1_{\cO_t}$, where $\cO_t$ are $\bfG^F_\nu$-orbits in $\bfE^F_\nu$ satisfying $\cO_t\not=\cO, \cO_t\preccurlyeq\cO$. By inductive hypothesis, each $1_{\cO_t}$ is a $\bbC$-linear combination of $\chi_{(A_s,\varphi_s)}$ for some $(A_s,\varphi_s)\in \tilde{\cD}_{\nu}$, then so is $1_\cO$, as desired.
\end{proof}

As a consequence, by (\ref{tilde(Delta)alghom}) and Proposition \ref{preserve structure constant}, for any $M,N\in \cA$, we have 
\begin{align*}
&\tilde{\Delta}(1_{\cO_M}\,\,\tilde{*}\,\,1_{\cO_N})=\tilde{\Delta}(1_{\cO_M})\,\,\tilde{*}\,\,\tilde{\Delta}(1_{\cO_N}),\\
&\Delta(u_{[M]}*u_{[N]})=\Delta(u_{[M]})*\Delta(u_{[N]}),
\end{align*}
which is equivalent to Green's formula for $\cA=\rep_k^{F_{Q,a}}(Q)$, see Theorem \ref{Green's formula}. 

Moreover, let $\Lambda$ be a finite-dimensional hereditary basic algebra over $\bbF_q$, we have the bialgebra isomorphism $\cH(\mod_{\bbF_q}\Lambda)\cong \cH(\rep_k^{F_{Q_\Lambda,a}}(Q_\Lambda))$, see section \ref{Hall algebra}, and so we also obtain Green's formula for $\mod_{\bbF_q}\Lambda$. 

\subsection*{Acknowledgement} 
J. Fang is partially supported by National Key R\&D Program of China (Grant No. 2020YFE0204200). Y. Lan is partially supported by National Natural Science Foundation of China (Grant No. 12288201). Y. Wu is partially supported by Natural Science Foundation of China (Grant No. 12031007). The authors would like to thank Jie Xiao for important discussions and helpful suggestions.

\bibliography{ref}

\end{document}